\documentclass[10pt,a4paper]{article}
\usepackage{url}
\usepackage{tikz}
\usetikzlibrary{arrows}
\usepackage[english]{babel}
\usepackage[caption=false]{subfig}
\usepackage[T1]{fontenc}
\usepackage{amsmath,amssymb}
\usepackage{eucal}
\usepackage{bm}
\usepackage{esint}
\usepackage{natbib}
\usepackage[margin=3.5cm]{geometry}
  
\usepackage{amsthm}
\newtheorem{theorem}{Theorem}
\newtheorem{remark}[theorem]{Remark}
\newtheorem{algorithm}[theorem]{Algorithm}
\newtheorem{lemma}[theorem]{Lemma}
\newtheorem{corollary}[theorem]{Corollary}
\newtheorem{proposition}[theorem]{Proposition}

\numberwithin{theorem}{section}
\numberwithin{equation}{section}

\newcommand{\tri}{\mathcal{T}}
\newcommand{\faces}{\mathcal{F}}
\newcommand{\norm}[1]{\lVert#1\rVert}
\newcommand{\abs}[1]{\lvert#1\rvert}
\newcommand{\ennorm}[1]{\lvert\!\lvert\!\lvert #1 \rvert\!\rvert\!\rvert}
\newcommand{\hnull}{\norm{h_0}_\infty}
\newcommand{\nc}{{\textup{\tiny NC}}}
\newcommand{\ennormnc}[1]{\ennorm{#1}_\nc}
\newcommand{\Curl}{\operatorname{Curl}}
\newcommand{\osc}{\operatorname{osc}}
\newcommand{\card}{\operatorname{card}}
\newcommand{\Imorl}{\mathcal{I}}

\usepackage{fancyhdr}
\begin{document}

\title{Morley Finite Element Method for the Eigenvalues
       of the Biharmonic Operator}
\author{Dietmar Gallistl
        \thanks{Institut f\"ur Numerische Simulation, 
        Universit\"at Bonn,
         Wegelerstra{\ss}e 6,
            D-53115 Bonn, Germany,
           \texttt{gallistl@ins.uni-bonn.de}.
       \newline
       The author was supported by the DFG Research Center Matheon, Berlin.}}
\date{}
\maketitle

\begin{abstract}
This paper studies the nonconforming Morley finite element 
approximation of the eigenvalues of the biharmonic operator.
A new $C^1$ conforming companion operator leads to an $L^2$ error 
estimate for the Morley finite element method which directly compares the $L^2$
error with the error in the energy norm and, hence, can dispense with any
additional regularity assumptions.
Furthermore, the paper presents new eigenvalue error estimates for
nonconforming finite elements that bound the 
error of
(possibly multiple or clustered) eigenvalues 
by the 
approximation error of the computed invariant subspace.
An application is the proof of  optimal convergence rates
for the adaptive Morley finite element method for
eigenvalue clusters.
\end{abstract}

\noindent
{\footnotesize
 {\bf Keywords}
 eigenvalue problem, eigenvalue cluster, 
 Kirchhoff plate, biharmonic, Morley,
 adaptive finite element method}

\noindent
{\footnotesize
 {\bf AMS subject classifications} 65M12, 65M60, 65N25}

\section{Introduction}

Let $\Omega\subseteq \mathbb R^2$ be 
an open bounded 
Lipschitz domain with polygonal boundary $\partial\Omega$ and 
outer unit normal $\nu$.
The boundary is decomposed into mutually disjoint parts
\begin{equation*}
\partial\Omega = \Gamma_C \cup \Gamma_S \cup \Gamma_F
\end{equation*}
such that $\Gamma_C$ and $\Gamma_C\cup\Gamma_S$ are closed sets.
The vector space of admissible functions reads as
\begin{equation*}
 V:=\left\{v\in H^2(\Omega)\bigm| v|_{\Gamma_C\cup\Gamma_S}=0
           \text{ and }(\partial v/\partial\nu)|_{\Gamma_C} = 0
    \right\}.
\end{equation*}
The biharmonic eigenvalue problem seeks eigenpairs 
$(\lambda,u)\in\mathbb R \times V$ with
\begin{equation}\label{e:EVPclassic}
  (D^2 u, D^2 v)_{L^2(\Omega)} =  \lambda (u,v)_{L^2(\Omega)}
  \quad\text{for all } v\in V.
\end{equation}
In the Kirchhoff-Love plate model \citep{TimoshenkoGere1985},
the problem \eqref{e:EVPclassic} describes the vibrations of a
thin elastic plate subject to clamped ($\Gamma_C$),
simply supported ($\Gamma_S$) or free ($\Gamma_F$)
boundary conditions.
Nonconforming finite element discretisations of \eqref{e:EVPclassic}
appear attractive because they circumvent the use of complicated
$C^1$ conforming FEMs \citep{Ciarlet1978}. 
The nonconforming Morley finite element based on piecewise quadratic
polynomials can furthermore be employed for the computation of lower
eigenvalue bounds \citep{CarstensenGallistl2014}. For the linear 
biharmonic problem, the adaptive Morley FEM has been proven to produce
optimal convergence rates 
\citep{HuShiXu2012Morley,CarstensenGallistlHu2014}.

A~priori error estimates for the Morley finite element discretisation
of eigenvalue problems can be found in \citep{Rannacher1979}.
In the a~posteriori error analysis, in particular for the analysis
of adaptive algorithms, the $L^2$ error of the eigenfunction 
approximation can be viewed as a perturbation of the right-hand side.
Indeed, for conforming finite elements, the higher-order $L^2$ error
control follows from the Aubin-Nitsche duality technique
\citep{StrangFix1973}. This argument fails to hold in its original
form in the case of nonconforming finite elements. In order to obtain
error estimates in the $L^2$ norm that do not require additional
assumptions on the regularity of the solution, the works
\citep{CarstensenGallistlSchedensack2014,MaoShi2010} introduced
(for the Crouzeix-Raviart discretisation of second-order problems)
certain conforming companion operators that allow the proof of such
$L^2$ estimates. This paper introduces a corresponding operator
for the Morley finite element. This operator leads to a new $L^2$ error
estimate for the Morley finite element without any additional regularity
assumption. This is of particular interest in the case of non-clamped
boundary conditions where, in general, the exact solution is expected to
belong to $H^2(\Omega)\setminus H^{5/2}(\Omega)$.

Practical adaptive algorithms for multiple eigenvalues
\citep{DaiHeZhou2012v2} or eigenvalue clusters
\citep{Gallistl2014nc,Gallistl2014} are based on a~posteriori error 
estimators that involve the sum of the residuals of all discrete
eigenfunctions of interest. 
Let $\lambda_{n+1}\leq\dots\leq\lambda_{n+N}$ be the eigenvalue cluster
of interest with discrete approximations
$\lambda_{\ell,n+1}\leq\dots\leq\lambda_{\ell,n+N}$ computed by the
Morley FEM.
These error estimators bound the distance
of the exact invariant subspace of the corresponding eigenfunctions
$W=\operatorname{span}\{u_{n+1},\dots,u_{n+N}\}$
and the invariant subspace of discrete eigenfunctions
$W_\ell=\operatorname{span}\{u_{\ell,n+1},\dots,u_{\ell,n+N}\}$.
 For conforming finite elements, the 
results of \citet{KnyazevOsborn2006} show that this distance acts as
an upper bound of the eigenvalue error. This result, however, does
not directly apply to nonconforming finite element methods.
A generalisation for the Crouzeix-Raviart FEM for the eigenvalues of
the Laplacian is given in \citep{BoffiDuranGardiniGastaldi2014} where
it is used that the nonconforming finite element space has an 
$H^1$-conforming subspace. The Morley finite element does not satisfy
a corresponding condition; this paper develops a new technique which
allows the proof of eigenvalue error estimates of the form 
$$
 \lvert \lambda_j - \lambda_{\ell,j}\rvert 
  \big/ \max\{\lambda_j,\lambda_{\ell,j}\}
 \leq C \sin_{a,\nc}^2 (W,W_\ell).
$$
The constant $C$ and its dependence on the eigenvalue cluster will be
quantified more precisely. The angles are measured in the discrete
energy scalar product ($L^2$ product of the piecewise Hessians).
 The main idea is to study an auxiliary
eigenvalue problem in the sum 
$\widehat V_\ell := V + V_\ell$ of the continuous space $V$ and the
discrete space $V_\ell$.  
The arguments in the proof rely
on a careful analysis of the Morley interpolation operator and the
conforming companion operator.

As an application, the paper presents optimal convergence rates of
the adaptive Morley FEM for eigenvalue clusters. The proofs follow
the methodology of \citep{CKNS08,Stevenson2007} which has already been
applied in
\citep{DaiXuZhou2008,CarstensenGedicke2012,CarstensenGallistlSchedensack2014} 
for simple eigenvalues,
in \citep{DaiHeZhou2012v2} for multiple eigenvalues, and in
\citep{Gallistl2014nc,Gallistl2014} for clustered eigenvalues.

\medskip
The remaining parts of this paper are organised as follows.
Section~\ref{s:MorleyFEM} introduces the necessary notation on
triangulations and data structures, it proves new error estimates for
the Morley interpolation operator, and it presents a new conforming
companion operator.
Section~\ref{s:evalest} is devoted to the discretisation of the 
biharmonic eigenvalue problem and derives
new $L^2$ error estimates and new error estimates for the eigenvalues
whose proof is based on a new methodology.
Section~\ref{s:AFEM} applies the new results to the adaptive
finite element method for clustered eigenvalues and proves its
optimal convergence rates.

\bigskip
Throughout the paper standard notation on Lebesgue and
Sobolev spaces is employed.
The integral mean is denoted by $\fint$. 
The bullet $\bullet$ denotes the identity.
For any smooth function $f:\Omega\to\mathbb R$ the Curl reads as
$
  \Curl f := (-\partial f /\partial x_2,\;\partial f / \partial x_1) .
$
For a sufficiently smooth vector field
$\beta: \Omega\to\mathbb R^2$, define
\begin{equation*}
 \Curl \beta := \begin{pmatrix} 
           -\partial\beta_1/\partial x_2&\partial\beta_1/\partial x_1\\
            -\partial\beta_2/\partial x_2&\partial\beta_2/\partial x_1
            \end{pmatrix} .         
\end{equation*}
The symmetric part of a matrix $X$ is denoted by 
$\operatorname{sym}(X)$ and the space of symmetric $2\times 2$
matrices is denoted by $\mathbb{S}$.
The notation $a\lesssim b$
abbreviates $a\leq C b$ for a positive generic constant $C$
that may depend on the domain $\Omega$ and the
initial triangulation $\tri_0$ but not on the mesh-size or the 
eigenvalue cluster of interest.
The notation
$a\approx b$ stands for $a\lesssim b\lesssim a$.

\

\section{The Morley Finite Element Space}
 \label{s:MorleyFEM}

This section introduces the necessary notation and data structures
in Subsection~\ref{ss:datastruct} and proves some new results
for the Morley finite element in the remaining subsections.

\subsection{Notation and Data Structures}\label{ss:datastruct}

\paragraph{Triangulations.}
Let $\tri_0$ be a regular triangulation of $\Omega$,
 i.e., $\cup\tri_0 = \overline\Omega$ and any two distinct elements
 of $\tri_0$ are either disjoint or their intersection is
 exactly one common vertex or exactly one common edge.
 Throughout this paper, any regular triangulation of
 $\Omega$ is assumed to be admissible in the sense that it
 is regular and a refinement of some initial triangulation $\tri_0$
 created by newest-vertex bisection with proper initialisation 
 of the refinement edges  \citep{BinevDahmenDeVore2004,Stevenson2008}.
 The set of all admissible  refinements is denoted by $\mathbb T$.
 The restriction to this class of triangulations is not essential 
 in Sections~\ref{s:MorleyFEM}--\ref{s:evalest}, but is made to ease
 notation in view of the adaptive algorithms studied in 
 Section~\ref{s:AFEM}.
 Given a triangulation $\tri_\ell\in\mathbb T$,
 the piecewise constant
 mesh-size function $h_\ell:=h_{\tri_\ell}$ is defined by
 $h_\ell|_T := h_T := \operatorname{meas}(T)^{1/2}$ for any
 triangle $T\in\tri_\ell$.
 For all regular triangulations $\tri_\ell\in\mathbb T$ 
of $\Omega$, it is assumed that the relative interior of
each boundary edge is contained in one of the parts
$\Gamma_C$, $\Gamma_S$, or $\Gamma_F$ (in fact, this is
only a condition on $\tri_0$).

 \paragraph{Edges.}
 The set of edges of a triangle $T$ is denoted by $\faces(T)$.
 The edges of $\tri_\ell$ read as
 $\faces_\ell:=\faces(\tri_\ell):= \cup_{T\in\tri_\ell} \faces(T)$.
 The edges that belong to the boundary read
 $\faces_\ell(\partial\Omega)$ and the interior edges read
 $\faces_\ell(\Omega):= \faces_\ell \setminus \faces_\ell(\partial\Omega)$.
 Let $\Gamma\subseteq\partial\Omega$ be a subset of the boundary
 $\partial\Omega$. The boundary edges that belong to $\Gamma$
 are denoted by
 $
 \faces_\ell(\Gamma)
 :=\{F\in\faces_\ell\mid \mathcal{H}^1(F\cap \Gamma)>0\}
$,
where $\mathcal{H}^1$ is the one-dimensional Hausdorff measure.
Furthermore, define
$\faces_\ell(\Omega\cup\Gamma) := \faces_\ell(\Omega) \cup\faces_\ell(\Gamma)$.
For any edge $F\in\faces_\ell$, the edge patch is defined as
$\omega_F:=\operatorname{int}(\cup\{T\in\tri_\ell\mid F\in\faces(T)\})$.
Given any vertex of $\tri_\ell$, the set of edges that share $z$
is denoted by $\faces_\ell(z):=\{F\in\faces_\ell\mid z\in F\}$.
The length of an edge $F$ reads $h_F$.

\paragraph{Vertices.}
The set of vertices of a triangle $T$ is denoted by 
$\mathcal N (T)$.
Define
$\mathcal N_\ell:=\mathcal N(\tri_\ell)
:= \cup_{T\in\tri_\ell}\mathcal N(T)$
as the set of vertices of $\tri_\ell$.
The set of vertices that belong to some subset $\omega\subseteq\Omega$
is denoted by
$\mathcal N_\ell(\omega) := \mathcal N_\ell \cap \omega$.

\paragraph{Normal and tangent vectors.}
 Let every edge
 $F\in\faces_\ell$ be equipped with a fixed normal vector 
 $\nu_F$.  
 If $F\in\faces_\ell(\partial\Omega)$ belongs to the boundary,
 $\nu_F:=\nu$ is chosen to point outwards $\Omega$.
Let for any edge $F\in\faces_\ell$ with normal vector
$\nu_F =(\nu_F(1);\nu_F(2))$ the tangent vector
be defined as $\tau_F:=(-\nu_F(2);\nu_F(1))$ and
denote by $\tau := (-\nu(2);\nu(1))$ the tangent vector
of $\partial \Omega$.

\paragraph{Jumps.}
Given $F\in\faces_\ell(\Omega)$, 
 $F=\partial T_+ \cap \partial T_-$ shared by two triangles
 $(T_+, T_-)\in\tri_\ell^2$, and a piecewise (possibly vector-valued)
 smooth function $v$, define the jump of $v$ across $F$ by
 \begin{equation*}
   \left[v\right]_F := v|_{T_+} - v|_{T_-}.
 \end{equation*}
 For edges $F\subseteq\partial\Omega$ on the boundary,
 $[v]_F:=v|_F$ denotes the trace.

 \paragraph{Piecewise polynomials and oscillations.}
The set of polynomials of degree $\leq k$ over a subset
$\omega\subseteq\Omega$ is denoted by $\mathcal P_k(\omega)$.
  The set of piecewise polynomial functions of degree
  $\leq k$ with respect to $\tri_\ell$ is denoted by 
  $\mathcal{P}_k(\tri_\ell)$.
 The $L^2$ projection onto 
 $\mathcal{P}_k(\tri_\ell)$ is denoted by
 $\Pi_{\tri_\ell}^k \equiv \Pi_\ell^k$.
 The $k$-th order oscillations of a given function $f\in L^2(\Omega)$
is defined as
\begin{equation*}
 \osc_k(f,\tri_\ell) := \| h_\ell^2 (1-\Pi_\ell^k) f \|_{L^2(\Omega)}.
\end{equation*}

\paragraph{Piecewise action of differential operators.}
The piecewise action of a differential operator is indicated by the 
subscript NC,  i.e., the piecewise versions of
 $D$ and $D^2$ read as
 $D_{\nc}\equiv D_{\nc(\tri_\ell)}$ and
 $D_{\nc}^2\equiv D^2_{\nc(\tri_\ell)}$,
  e.g., $(D_\nc v)|_T = D(v|_T)$ for any $T\in\tri_\ell$.
 The dependence on $\tri_\ell$ in this notation is dropped whenever
 there is no risk of confusion.

\paragraph{Functional setting.}
The vector space of admissible functions reads as
\begin{equation*}
 V:=\left\{v\in H^2(\Omega)\bigm| v|_{\Gamma_C\cup\Gamma_S}=0
           \text{ and }(\partial v/\partial\nu)|_{\Gamma_C} = 0
    \right\}.
\end{equation*}
Define the bilinear form 
\begin{equation*}
  a(v,w) : =  ( D^2 v , D^2 w)_{L^2(\Omega)}
  \quad \text{for all }  (v,w)\in V^2
\end{equation*}
with induced seminorm $\ennorm{\cdot} := a(\cdot,\cdot)^{1/2}$
and $b(\cdot,\cdot):=(\cdot,\cdot)_{L^2(\Omega)}$ with induced
norm $\norm{\cdot}$.
Throughout this paper it is assumed that the only affine
function in $V$ is zero, i.e., $V\cap\mathcal P_1(\Omega) = \{0\}$.
Hence, $a$ is a scalar product on $V$ with norm $\ennorm{\cdot}$.

The Morley finite element space reads as
\begin{align*}
 V_\ell :=
 \left\{  v  \in \mathcal P_2(\tri_\ell)
  \left|
   \begin{array}{l}
   v \text{ is continuous at } \mathcal N_\ell(\Omega)
   \text{ and vanishes at } \mathcal N_\ell(\Gamma_C \cup \Gamma_S);\\
   D_\nc v
    \text{ is continuous at the interior edges' midpoints}\\
   \text{and vanishes at the midpoints of the edges of } \Gamma_C
   \end{array}
   \right.
  \right\}   .
\end{align*}
On each triangle the local degrees of freedom are the evaluation
of the function at each vertex and the evaluation of the normal
derivative at the edges' midpoints. See Figure~\ref{sf:FEMs:Morley}
for an illustration.

The discrete version of the energy scalar product reads as
\begin{equation*}
  a_\nc(v,w) := ( D^2_\nc v , D^2_\nc w)_{L^2(\Omega)}
  \quad\text{for all }
  (v,w) \in ( V + V_\ell)^2
\end{equation*}
with induced discrete energy norm
$\ennormnc{\cdot} := a_\nc(\cdot,\cdot)^{1/2}$.
Indeed, the assumption $V\cap\mathcal P_1(\Omega) =\{0\}$
implies $V_\ell\cap\mathcal P_1(\Omega) =\{0\}$.
Hence, $a_\nc(\cdot,\cdot)$ defines a scalar product on $V_\ell$
(as shown in Corollary~\ref{c:dFMorley}, the ellipticity is
 is even uniform in the mesh parameter).

\paragraph{Principal angles between subspaces.}
For finite-dimensional subspaces
$X\subseteq V + V_\ell$ and
$Y\subseteq V + V_\ell$,
the sine of the largest principal angle from $X$
to $Y$ is denoted by
\begin{equation*}
\sin_{a,\nc}\angle(X,Y)
=
\sup_{\substack{x\in X \\ \ennormnc{x}=1}}
   \inf_{y\in Y} \ennormnc{x-y} .
\end{equation*}
It is well known \citep[Thm.\ 6.34 in Chapter 1, \S 6]{Kato1966}
that in the case of $\dim(X)=\dim(Y)<\infty$ it holds that
\begin{equation}\label{e:SinVertausch}
\sin_{a,\nc}\angle(X,Y) = \sin_{a,\nc}\angle(Y,X)
\end{equation}
as well as
\begin{equation}\label{e:SinDreiecksungl}
\sin_{a,\nc}(X,Y) 
\leq \sin_{a,\nc}\angle(X,Z)+\sin_{a,\nc}\angle(Z,Y)
\end{equation}
for any subspace $Z\subseteq V +V_\ell$ with
$\dim(X)=\dim(Y)=\dim(Z)<\infty$.

\subsection{Morley Interpolation Operator}

Let $\tri_{\ell+m}$ be any admissible refinement 
of $\tri_\ell$.
The Morley interpolation operator
$\Imorl_\ell: V + V_{\ell+m} \to V_\ell$
is defined via
  \begin{equation*}
   \begin{aligned}
    (\Imorl_\ell v ) (z) 
    &
    = v (z)
    &&\text{for any }z\in\mathcal N_\ell 
          \text{ and any } v\in V +V_{\ell+m}     ,
    \\
    \int_F \frac{\partial \Imorl_\ell v}{ \partial \nu_F }\,ds  
    &
    = \int_F \frac{\partial v}{ \partial \nu_F}\,ds
    &&\text{for any }F\in\faces_\ell 
          \text{ and any } v\in V+V_{\ell+m} .
   \end{aligned}
  \end{equation*}
  A piecewise integration by parts proves the
  projection property for the Hessian
  \begin{equation}\label{e:MorleyInterpolProjProp}
    \Pi_\ell^0 D^2_\nc = D^2_\nc \Imorl_\ell.
  \end{equation}

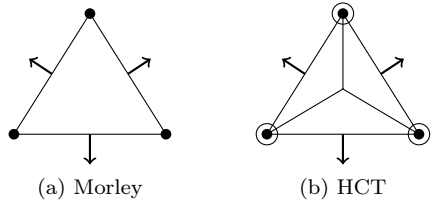
\begin{figure}[tb]
\begin{center}
\subfloat[Morley\label{sf:FEMs:Morley}]{
\begin{tikzpicture}[scale=2]
 \draw (0,0)--(1,0)--(.5,.8)--cycle;
 \foreach \x/\y  in {0/0,1/0,.5/.8}
      {  \fill (\x,\y) circle (1pt);}
 \foreach \a/\b/\c/\d  in {.5/0/.5/-.2,
                     .75/.4/.9/.5,
                      .25/.4/.1/.5}
      {\draw[->,thick] (\a,\b)--(\c,\d);}
\end{tikzpicture}
}
\hspace{5ex}
\subfloat[HCT\label{sf:FEMs:HCT}]{
\begin{tikzpicture}[scale=2]
 \draw (0,0)--(1,0)--(.5,.8)--cycle
       --(.5,.3)--(1,0)
       (.5,.3)--(.5,.8);
 \foreach \x/\y  in {0/0,1/0,.5/.8}
      {  \fill (\x,\y) circle (1pt);
         \draw (\x,\y) circle (2pt);}
 \foreach \a/\b/\c/\d  in {.5/0/.5/-.2,
                     .75/.4/.9/.5,
                      .25/.4/.1/.5}
      {\draw[->,thick] (\a,\b)--(\c,\d);}
\end{tikzpicture}
}
\end{center}
\caption{Mnemonic diagrams of the Morley (left) and the
         HCT (right) finite element.\label{f:FEMs}}
\end{figure}

The following generalisation of the trace inequality
\citep{CarstensenFunken2000,DiPietroErn2012} is necessary for proving 
error estimates for the Morley interpolation operator.

\begin{proposition}[discrete trace inequality]\label{p:DiscrTrace}
Let $T\in\tri_\ell$ be a triangle and $\mathcal K$ be a regular
triangulation of $T$ and let $G\in\faces(T)$ be an edge of $T$.
Any piecewise (with respect to $\mathcal K$) smooth function
$f$ satisfies the discrete trace inequality
\begin{equation*}
 \norm{f}_{L^2(G)}
\lesssim
 h_T^{-1/2} \norm{f}_{L^2(T)}
 + h_T^{1/2} \norm{D_{\nc(\mathcal K)} f}_{L^2(T)}
 + h_T^{1/2}\sqrt{
   \sum_{\substack{F\in\faces(\mathcal K) 
                   \\ F\not\subseteq \partial T}}
   h_F^{-1}\norm{[f]_F}_{L^2(F)}^2   }  .
\end{equation*}

\end{proposition}

\begin{proof}
 Denote by $P_G$ the vertex of $T$ opposite to $G$.
A piecewise integration by parts proves the discrete trace identity
\begin{equation*}
\frac12 \int_T (\bullet - P_G)\cdot D_{\nc(\mathcal K)} f \,dx
=
-\int_T f \,dx
+ \operatorname{dist}(P_G,G) \int_G f \,ds
+ \sum_{\substack{F\in\faces(\mathcal K) 
                   \\  F\not\subseteq \partial T}}
  \int_F (\bullet - P_G)\cdot\nu_F [f]_F\,ds .
\end{equation*}
The application of this identity to the function $f^2$
together with elementary algebraic manipulations and
$\operatorname{dist}(P_G,G)\leq \operatorname{diam}(T)\lesssim h_T$
result in
\begin{equation}\label{e:DisTracefSq}
\norm{f}_{L^2(G)}^2
\lesssim
\left| \int_T D_{\nc(\mathcal K)} (f^2) \,dx \right|
+
h_T^{-1} \norm{f}_{L^2(T)}^2
+ 
h_T^{-1}\sum_{\substack{F\in\faces(\mathcal K) 
                   \\  F\not\subseteq \partial T}}
  \int_F (\bullet - P_G)\cdot\nu_F [f^2]_F\,ds .
\end{equation}
The Young inequality
shows that the first term on the right-hand side can be controlled
as
\begin{equation*}
\begin{aligned}
\left| \int_T D_{\nc(\mathcal K)} (f^2) \,dx \right|
=
\left| \int_T 2 f D_{\nc(\mathcal K)} f \,dx \right|
&
\leq
2
h_T^{-1/2}
\norm{ f}_{L^2(T)}
h_T^{1/2}  \norm{D_{\nc(\mathcal K)} f }_{L^2(T)}
\\
&
\leq 
h_T^{-1}
\norm{ f}_{L^2(T)}^2
+
h_T  \norm{D_{\nc(\mathcal K)} f }_{L^2(T)}^2.
\end{aligned}
\end{equation*}
It remains to bound  the third term on the right-hand side
of \eqref{e:DisTracefSq}.
Let $F\in\faces(\mathcal K)$ be an interior edge shared by two
triangles $K_+$ and $K_-$ such that $F=K_+\cap K_-$.
Denote $f_+ := f|_{K_+}$ and $f_-:= f|_{K_-}$.
A direct calculation proves for the jump of $f^2$ across $F$ that
\begin{equation*}
 [f^2]_F = [f]_F (f_+ + f_-).
\end{equation*}
Thus, the Cauchy and triangle inequalities followed by
the Young inequality prove
\begin{equation*}
\begin{aligned}
 &\int_F (\bullet - P_G)\cdot\nu_F [f^2]_F\,ds 
 \\
 &
 \qquad\leq
 \operatorname{diam}(T) h_F^{-1/2}h_T^{1/2} \norm{[f]_F}_{L^2(F)}
           h_F^{1/2}h_T^{-1/2}
                (\norm{f_+}_{L^2(F)} +  \norm{f_-}_{L^2(F)})
 \\
 &               
 \qquad\leq              
 \operatorname{diam}(T) 
     \left(h_F^{-1}h_T \norm{[f]_F}_{L^2(F)}^2
           + h_F h_T^{-1}
                (\norm{f_+}_{L^2(F)} +  \norm{f_-}_{L^2(F)}) ^2
      \right).
\end{aligned}
\end{equation*}
The trace inequality \citep{CarstensenFunken2000,DiPietroErn2012}
and an inverse estimate \citep{BrennerScott2008}
applied to the edge patch $\omega_F$ prove that
\begin{equation*}
  h_F h_T^{-1}  (\norm{f_+}_{L^2(F)} +  \norm{f_-}_{L^2(F)}) ^2
 \lesssim
 h_T^{-1} \norm{f}_{L^2(\omega_F)}^2.
\end{equation*}
The foregoing two displayed inequalities, 
the finite overlap of the edge patches and the shape regularity prove
\begin{equation*}
 h_T^{-1}\sum_{\substack{F\in\faces(\mathcal K) 
                   \\  F\not\subseteq \partial T}}
  \int_F (\bullet - P_G)\cdot\nu_F [f^2]_F\,ds
\lesssim
 h_T^{-1} \norm{f}_{L^2(T)}^2
+
h_T
\sum_{\substack{F\in\faces(\mathcal K) 
                   \\  F\not\subseteq \partial T}}
  h_F^{-1} \norm{[f]_F}_{L^2(F)}^2
.
\end{equation*}
The combination of the above estimates concludes the proof.
\end{proof}

\begin{remark}
 In Proposition~\ref{p:DiscrTrace},
 the ratio $h_T/h_F$ is not required to be uniformly bounded.
\end{remark}

The next proposition provides an error estimate for the
Morley interpolation operator.
In contrast to the estimate from \citep{CarstensenGallistl2014} with
an explicit constant for the Morley interpolation when applied
to an $H^2$ function,
the following result gives an estimate for
more general piecewise smooth functions.

\begin{proposition}[error estimate for the Morley interpolation]
                   \label{p:IlEstimate}
Let $T\in\tri_\ell$ be a triangle, and let $\tri_{\ell+m}$
be a regular triangulation of $T$.
Any $v_{\ell+m}\in V+V_{\ell+m}$ and its interpolation
$\Imorl_{\ell} v_{\ell+m}$ satisfy
\begin{equation}\label{e:MorleyInterpolApprxStab}
   \begin{aligned}
      \lVert h_T^{-2} (1-\Imorl_\ell)v_{\ell+m} \rVert_{L^2(T)}
   &+ \lVert h_T^{-1} D_\nc(1-\Imorl_\ell) v_{\ell+m} \rVert_{L^2(T)}
 \\
  &
  \qquad\qquad\qquad\qquad
   \lesssim
    \lVert D^2_\nc (1-\Imorl_\ell)v_{\ell+m} \rVert_{L^2(T)}.
   \end{aligned}
  \end{equation}
\end{proposition}

\begin{remark}
 Error estimates of this type are stated and utilised in
 \citep{HuShiXu2012Morley} with a proof based on equivalence of
 norms.
 To make the constant in the estimate more transparent, a new
 proof is given here.
 It shall be pointed out that the
 constant in the assertion of Proposition~\ref{p:IlEstimate}
 does not depend on the triangulation $\tri_{\ell+m}$.
\end{remark}

\begin{proof}[Proof of Proposition~\ref{p:IlEstimate}]
Let, without loss of generality, 
         $v_{\ell+m}\in H^4(\operatorname{int}(T))+ V_{\ell+m}$
         (the general case then follows with a density argument).
The discrete Friedrichs inequality
\citep[Thm.~10.6.12]{BrennerScott2008} together with a scaling
argument
and the fact that $\Imorl_\ell v_{\ell+m}$ is continuous on $T$
yield that
\begin{equation*}
 \begin{aligned}
 \norm{(1-\Imorl_\ell)v_{\ell+m}}_{L^2(T)} ^2
\lesssim
\left\vert\int_{\partial T} (1-\Imorl_\ell)v_{\ell+m} \,ds \right\vert^2
&
+
h_T^2
\sum_{\substack{F\in\faces(\tri_{\ell+m})
       \\  F\not\subseteq \partial T}}
     h_F^{-1} \norm{[v_{\ell+m}]_F}_{L^2(F)}^2
\\
&
+  \norm{ h_T D_\nc (1-\Imorl_\ell)v_{\ell+m}}_{L^2(T)}^2  .
\end{aligned}
\end{equation*}
For any edge $G\in\faces(T)$, the H\"older and Friedrichs inequalities
prove that
\begin{equation*}
 \begin{aligned}
 \left\vert\int_{G} (1-\Imorl_\ell)v_{\ell+m} \,ds \right\vert
 & 
 \lesssim 
 h_G^{1/2} \norm{(1-\Imorl_\ell)v_{\ell+m}}_{L^2(G)}
 \\
 &
 \lesssim
 h_G^{3/2} \norm{\partial (1-\Imorl_\ell)v_{\ell+m} / \partial\tau_G}_{L^2(G)} .
 \end{aligned}
\end{equation*}
(Note that $v_{\ell+m}$ is differentiable
  and continuous along $G$.)
The discrete trace inequality from Proposition~\ref{p:DiscrTrace}
proves that this is controlled by some constant times
\begin{equation*}
\begin{aligned}
 h_T \norm{D_\nc (1-\Imorl_\ell)v_{\ell+m}}_{L^2(T)}
 &+ h_T^2 \norm{D_\nc^2 (1-\Imorl_\ell)v_{\ell+m}}_{L^2(T)}
 \\
 &+ h_T^2 \sqrt{
   \sum_{\substack{F\in\faces(\tri_{\ell+m}) 
                   \\  F\not\subseteq \partial T}}
   h_F^{-1}\norm{[D_\nc v_{\ell+m}]_F}_{L^2(F)}^2   }  .
\end{aligned}
\end{equation*}
For any face $F\in\faces(\tri_{\ell+m})$
with $ F\not\subseteq \partial T$,
the Friedrichs and Poincar\'e inequality prove that
\begin{equation*}
 h_F^{-1}\norm{[v_{\ell+m}]_F}_{L^2(F)}^2 
 \lesssim
 h_F \norm{[D_\nc v_{\ell+m}]_F\tau_F}_{L^2(F)}^2 
 \lesssim
  h_F^3 \norm{[D^2_\nc v_{\ell+m}]_F\tau_F}_{L^2(F)}^2 .
\end{equation*}
Altogether,
\begin{equation*}
 \begin{aligned}
 \norm{(1-\Imorl_\ell)v_{\ell+m}}_{L^2(T)}
&
\lesssim
 h_T \norm{D_\nc (1-\Imorl_\ell)v_{\ell+m}}_{L^2(T)}
 + h_T^2 \norm{D_\nc^2 (1-\Imorl_\ell)v_{\ell+m}}_{L^2(T)}
\\
&\qquad
 + h_T^{2} \sqrt{
   \sum_{\substack{F\in\faces(\tri_{\ell+m}) 
                   \\  F\not\subseteq \partial T}}
   h_F \norm{[D_\nc^2 v_{\ell+m}]_F \tau_F}_{L^2(F)}^2   }  .
 \end{aligned}
\end{equation*}
The discrete Friedrichs inequality
\citep[Thm.~10.6.12]{BrennerScott2008}
together with a scaling argument imply
\begin{equation*}
 h_T \norm{D_\nc (1-\Imorl_\ell)v_{\ell+m}}_{L^2(T)}
 \lesssim
 h_T^2 \norm{D_\nc^2 (1-\Imorl_\ell)v_{\ell+m}}_{L^2(T)} .
\end{equation*}

For the estimate of the jump terms
let $F=\operatorname{conv}\{z_1,z_2\}\in\faces(\tri_{\ell+m})$
be the convex hull of the vertices $z_1$, $z_2$ 
such that $F$ is an interior edge and denote, 
for $j\in\{1,2\}$, 
by $\varphi_j \in \mathcal P_1(\tri_{\ell+m})$ 
the piecewise affine function
with $\varphi_j(z_j)=1$ and $\varphi_j(y) = 0$ for all
$y\in\mathcal N(\tri_{\ell+m})\setminus\{z_j\}$.
The piecewise quadratic edge-bubble function
$\bm\flat_F := 6 \varphi_1\varphi_2\in H^1_0(\omega_F)$
satisfies
\begin{align*}
\norm{\bm\flat_F}_{L^\infty(T)} = 3/2
\quad\text{ and }\quad\int_F \bm\flat_F\,ds = h_F.
\end{align*}
Define 
$
 \psi_F:=(\bm\flat_F [D^2_\nc v_{\ell+m}]_F\tau_F)
 \in H_0^1(\omega_F;\mathbb R^2)
$.
Since $[D^2_\nc v_{\ell+m}]_F$ is constant along $F$, it follows that
\begin{equation*}
 \norm{[D^2_\nc  v_{\ell+m}]_F\tau_F}_{L^2(F)}^2
 =\norm{\bm\flat_F^{1/2}[D_\nc^2 v_{\ell+m}]_F\tau_F}_{L^2(F)}^2  .
\end{equation*}
For any $v\in H^2(\omega_F)$, an integration by parts and 
the $L^2$-orthogonality of $\Curl\psi_F$ on $D^2 v$
reveal that
\begin{align*}
 \norm{\bm\flat_F^{1/2}[D_\nc^2 v_{\ell+m}]_F\tau_F}_{L^2(F)}^2
 = \int_F \left([D_\nc^2 v_{\ell+m}]_F\tau_F\right)\cdot  \psi_F\,ds
 = (D_\nc^2 (v_{\ell+m} - v),\Curl\psi_F)_{L^2(\omega_F)} .
\end{align*}
The Cauchy and inverse inequalities prove that this is bounded by
\begin{align*}
\norm{D^2_\nc (v_{\ell+m} - v)}_{L^2(\omega_F)} 
           \norm{\Curl\psi_F}_{L^2(\omega_F)}
\lesssim \norm{D^2_\nc(v_{\ell+m}-v)}_{L^2(\omega_F)}
            \abs{ [D_\nc^2 v_{\ell+m}]_F\tau_F}.
\end{align*}
This implies
\begin{equation*}
 h_F \norm{[D^2_\nc v_{\ell+m}]_F\tau_F}_{L^2(F)} ^2
\lesssim
\min_{v\in H^2(\operatorname{int}(T))}
               \norm{D^2_\nc(v_{\ell+m}-v)}_{L^2(\omega_F)}^2. 
\end{equation*}
The sum over all interior edges of $\faces(\tri_{\ell+m})$
and the finite overlap of edge-patches prove the result.
\end{proof}

\subsection{Conforming Companion Operator}\label{ss:companion}
This subsection is devoted to the design of 
a new conforming companion operator. In contrast to the operators
introduced in \citep{CarstensenGallistlSchedensack2014,MaoShi2010},
$H^2$ conformity is required.
Compared to certain averaging operators that can be found
in the literature \citep{BrennerGudiSung2010,Gudi2010},
the proposed companion operator has additional conservation
properties for the integral mean and the integral mean of
the Hessian. 
A similar approach has been independently developed in
\citep{LiGuanMao2014}. In contrast to that work, the operator
presented here satisfies an additional best-approximation property.

The Hsieh-Clough-Tocher (HCT) finite element \citep{Ciarlet1978}
enters the design of a conforming companion operator.
Let any $T\in\tri_\ell$ be decomposed into three sub-triangles as
depicted in Figure~\ref{sf:FEMs:HCT}, where the vertex shared by
the three sub-triangles is the midpoint $\operatorname{mid}(T)$.
Given this triangulation $\mathcal K_\ell (T)$ of $T$, let
\begin{equation*}
 V_{\mathrm{HCT}}(\tri_\ell) :=
 \left\{ v\in V
 \left|
   v|_T \in \mathcal P_3(\mathcal K_\ell(T))
    \text{ for all } T\in\tri_\ell
 \right. 
 \right\} .
\end{equation*}
The local degrees of freedom on each triangle $T$ are the nodal
values of the function and its derivative and
the value of the normal derivative at the midpoints of 
the edges of $T$ in Figure~\ref{sf:FEMs:HCT}.

Such conforming
finite elements turn out to be useful for the theoretical analysis.
The following proposition presents a simple averaging operator,
similar to that of \citep{BrennerGudiSung2010,Gudi2010},
for the case of more general boundary contitions.

\begin{proposition}[HCT enrichment]\label{p:HCTenrichment}
There exists an operator 
$\mathcal A :V_\ell \to V_{\mathrm{HCT}}(\tri_\ell)$
such that any $v_\ell\in V_\ell$ satisfies
\begin{equation*}
 \begin{aligned}
 \norm{h_\ell^{-2}(v_\ell - \mathcal A v_\ell)}_{L^2(\Omega)}^2
 &
 \lesssim
 \sum_{F\in\faces_\ell(\Omega\cup\Gamma_C)} h_F \norm{[D^2 v_\ell]_F\tau_F}_{L^2(F)}^2
+
 \sum_{F\in\faces_\ell(\Gamma_S)} h_F \norm{\tau_F\cdot[D^2 v_\ell]_F\tau_F}_{L^2(F)}^2
 \\
 &
 \lesssim
 \min_{v\in V} \norm{D^2_\nc(v_\ell - v)}_{L^2(\Omega)} .
 \end{aligned}
\end{equation*}
\end{proposition}
\begin{proof}
Given $v_\ell \in V_\ell$, define 
$\mathcal A v_\ell \in V_{\mathrm{HCT}}(\tri_\ell)$ by setting the
degrees of freedom as follows
\begin{equation*}
 \begin{aligned}
   (v_\ell- \mathcal A v_\ell) (z) 
 &
  = 0
 &&\quad\text{for all } z\in\mathcal N_\ell,
  \\
  \frac{\partial(v_\ell -\mathcal A v_\ell)}{\partial\nu_F} (\operatorname{mid}(F))
  &
   = 0 
  &&\quad\text{for all } F\in\faces_\ell,
 \\
   D(\mathcal A v_\ell) (z) 
  &
    = \card(\tri_\ell(z))^{-1} \sum_{T\in\tri_\ell(z)}
      (D v_\ell|_T)(z)
  &&\quad\text{for all } z\in\mathcal N_\ell(\Omega\cup\Gamma_F).
 \end{aligned}
\end{equation*}
In other words, the degrees of freedom are defined by
averaging.
For the remaining vertices on the boundary, set
\begin{equation*}
D(\mathcal A v_\ell)(z) 
= 0 \quad \text{for all } z\in\mathcal N_\ell(\Gamma_S)
          \text{ with angle }  \neq \pi
          \text{ and all }z\in\mathcal N_\ell(\Gamma_C)
\end{equation*}
and, for all $z\in\mathcal N_\ell(\Gamma_S)$ with angle $= \pi$,
\begin{equation*}
\frac{\partial \mathcal A v_\ell}{\partial\tau}(z) 
= 0 
\quad\text{and}\quad
\frac{\partial \mathcal A v_\ell}{\partial\nu}(z)
= (\card(\tri_\ell(z)))^{-1}
    \sum_{F\in\{F_+,F_-\}}
             \left.\frac{\partial v_\ell}{\partial\nu(z)}\right|_F (z) 
\end{equation*}
where $(F_+,F_-)\in\faces_\ell(\Gamma_S)^2$ are the two boundary
edges sharing $z$.
Note that, for corners of the domain $\Omega$ with angle $\neq \pi$,
the simply supported boundary condition implies that the
full derivative vanishes at $z$.

The remaining part of the proof is devoted to the error estimate
for $\mathcal A$.
For a multi-index $\alpha$ of length $\abs{\alpha} =1$ and
any vertex $z\in\mathcal N_\ell$, let $\psi_{z,\alpha}$ 
denote the nodal basis function of $V_{\mathrm{HCT}}(\tri_\ell)$ with 
$(\partial\psi_{z,\alpha}/\partial x^{\alpha}) (z)=1$
that vanishes for the remaining degrees of freedom of the
HCT finite element.
Since the HCT finite element
is a finite element in the sense of \citet{Ciarlet1978},
for any $T\in\tri_\ell$ the function $v_\ell|_T\in \mathcal P_2(T)$
can be represented by means of the local HCT basis functions.
By definition of $\mathcal A$,
the difference $v_\ell - \mathcal A v_\ell$ can be represented
as follows
\begin{equation*}
  \norm{h_\ell^{-2}(v_\ell-\mathcal A v_\ell)}_{L^2(\Omega)}^2 
 = \sum_{T\in\tri_\ell}
   \bigg\Vert h_T^{-2}\sum_{z\in\mathcal N(T)} \sum_{\abs{\alpha}=1}
	\frac{\partial^{\abs{\alpha}} (v_\ell|_T - \mathcal A v_\ell)}
             {\partial x^\alpha} (z)
        \psi_{z,\alpha}
   \bigg\Vert _{L^2(T)}^2 .
\end{equation*}
For any $T\in\tri_\ell$, the scaling of the basis functions
\citep[Thm.~6.3.1, p.~344]{Ciarlet1978} reads as
\begin{equation*}
 \norm{h_T^{-2}\psi_{z,\alpha}}_{L^2(T)} \lesssim 1 
     \text{ for }\abs{\alpha} = 1.
\end{equation*}
Thus, the triangle inequality implies that
\begin{equation*}
  \norm{h_\ell^{-2}(v_\ell-\mathcal A v_\ell)}_{L^2(\Omega)}^2 
  \lesssim
   \sum_{T\in\tri_\ell}
    \sum_{z\in\mathcal N(T)} 
     \abs{ D (v_\ell|_T -\mathcal A v_\ell)(z)}^2 .
\end{equation*}
The triangle inequality and equivalence of seminorms prove, for any
vertex $z\in\mathcal N_\ell(\Omega\cup\Gamma_F)$, that
\begin{equation}\label{e:AverageTriEq}
 \abs{ D (v_\ell|_T -\mathcal A v_\ell)(z)} ^2
 \lesssim
 \sum_{F\in\faces_\ell(z)\cap\faces_\ell(\Omega)} [D_\nc v_\ell (z)]_F^2
 \lesssim
 \sum_{F\in\faces_\ell(z)\cap\faces_\ell(\Omega)}h_F^{-1} 
   \norm{[D_\nc v_\ell]_F}_{L^2(F)}^2  .
\end{equation}
For any vertex $z\in\mathcal N_\ell(\Gamma_C)$ and any triangle
$T$ with $z\in T$ the definition of $\mathcal A$ implies
\begin{equation*}
 \abs{ (D_\nc v_\ell|_T - \mathcal A v_\ell )(z)}
 =\abs{D v_\ell|_T(z)}.
\end{equation*}
Any vertex $z\in\mathcal N_\ell(\Gamma_S)$ and any triangle
$T$ with $z\in T$ satisfy
\begin{equation*}
 \abs{(\partial (v_\ell|_T - \mathcal A v_\ell) / \partial\tau )(z)}
=\abs{(\partial v_\ell |_T/\partial\tau )(z)}
\end{equation*}
and, as in \eqref{e:AverageTriEq}, it follows
in the case that the angle at $z$ equals $\pi$, that
\begin{equation*}
 \abs{(\partial (v_\ell|_T - \mathcal A v_\ell) / \partial\nu )(z)}
\lesssim
\sum_{F\in\faces_\ell(z)\cap\faces_\ell(\Omega)} 
  \abs{[\partial v_\ell /\partial \nu_F]_F (z)}   .
\end{equation*}
Equivalence of norms and
Poincar\'e inequalities along $F\in\faces_\ell$ prove
\begin{equation*}
\begin{aligned}
\abs{\left[\partial v_\ell/\partial\tau_F\right]_F(z)}
&
\lesssim h_F^{-1/2}\norm{\left[\partial v_\ell/\partial\tau_F\right]_F}_{L^2(F)}
\lesssim h_F^{1/2}\norm{\tau_F\cdot \left[D^2_\nc v_\ell\right]_F\tau_F}_{L^2(F)}
,\\
\abs{\left[\partial v_\ell/\partial\nu_E\right]_E(z)}
&
\lesssim h_F^{-1/2}\norm{\left[\partial v_\ell/\partial\nu_F\right]_F}_{L^2(F)}
\lesssim h_F^{1/2}\norm{\nu_F\cdot \left[D^2_\nc v_\ell\right]_F\tau_F}_{L^2(F)}
.
\end{aligned}
\end{equation*}
This proves the first inequality of the proposition.

The proof of the efficiency estimate can be carried out
by using the bubble function technique
from the proof of Proposition~\ref{p:IlEstimate}.
\end{proof}

\begin{proposition}[companion operator]\label{p:companion}
For any $v_\ell\in V_\ell$ there exists some
$
 \mathcal{C}v_\ell\in V
$
such that $v_\ell-\mathcal{C}v_\ell$ and its second-order
partial derivatives are $L^2$-orthogonal on
the space $\mathcal P_0(\tri_\ell)$ of piecewise constants,
\begin{align}\label{e:CompanionProj}
 \Pi_\ell^0 (v_\ell - \mathcal C v_\ell) = 0
 \quad\text{and}\quad
 \Pi_\ell^0 (D_\nc^2(v_\ell-\mathcal{C}v_\ell))=0.
\end{align}
Moreover, the operator $\mathcal C$%
satisfies the approximation and stability
property
\begin{equation}\label{e:CompanionApproxStab}
\begin{aligned}
&
\norm{h_\ell^{-2}  (v_\ell-\mathcal{C} v_\ell)}_{L^2(\Omega)}
+\norm{h_\ell^{-1} D_\nc  (v_\ell-\mathcal{C} v_\ell)}_{L^2(\Omega)}
+\norm{D^2_\nc (v_\ell-\mathcal{C} v_\ell)}_{L^2(\Omega)}
\\
& \qquad\qquad\qquad\qquad\qquad\qquad\qquad\qquad\qquad\qquad
\lesssim \min_{v\in V}\norm{D^2_\nc (v_\ell-v)}_{L^2(\Omega)}.
\end{aligned}
\end{equation}
\end{proposition}
\begin{proof}
 The design follows in three steps.
 
{\em Step 1.}
Proposition~\ref{p:HCTenrichment} 
and inverse estimates \citep{BrennerScott2008}
 prove for the operator $\mathcal A$ that
\begin{equation*}
\begin{aligned}
  \norm{h_\ell^{-2} (v_\ell-\mathcal A v_\ell)}_{L^2(\Omega)}
+ \norm{h_\ell^{-1} D_\nc (v_\ell-\mathcal A v_\ell)}_{L^2(\Omega)}
&
+ \norm{D^2_\nc(v_\ell-\mathcal A v_\ell)}_{L^2(\Omega)}
\\
 & \qquad\qquad
\lesssim \min_{v\in V}\norm{D^2_\nc(v_\ell-v)}_{L^2(\Omega)}.
\end{aligned}
\end{equation*}

{\em Step 2.}
Let $T=\operatorname{conv}\{z_1,z_2,z_3\}$ be a triangle of $\tri_\ell$
and let $F\in\faces(T)$ with $F=\operatorname{conv}\{z_1,z_2\}$ and
denote  the continuous nodal $\mathcal P_1$  basis functions by
$
\varphi_1,\varphi_2, \varphi_3
     \in \mathcal P_1(\tri_\ell)\cap H^1(\Omega)
$.
Let $\nu_T$ denote the outward pointing unit normal of $T$
and define the function $\zeta_{F,T}$ by
\begin{equation*}
\zeta_{F,T} :=  30 (\nu_T\cdot\nu_F)
      \operatorname{dist}(z_3,F)\varphi_1^2\varphi_2^2\varphi_3.
\end{equation*}
For any $F\in\faces_\ell$, the function
\begin{equation*}
 \zeta_F := \begin{cases}
             \zeta_{F,K} &\text{on triangles } K\in\tri_\ell 
                              \text{ with } F\in\faces(K),\\
             0 &\text{otherwise}
            \end{cases}
\end{equation*}
satisfies $\zeta_F \in H^2(\Omega)$ and 
$\operatorname{supp}(\zeta_F)=\overline{\omega_F}$
as well as $\fint_F \partial \zeta_F / \partial \nu_F\,dx =1$.
For the proof that $\zeta_F$ is continuously differentiable across
interior edges $F$, note that any adjacent triangle $T$ satisfies
$D\varphi_3|_T = (\operatorname{dist}(z_3,F))^{-1}\nu_T$ as well
as
\begin{equation*}
 (D \zeta_{F,T} )|_F \nu_F 
 = 30 (\nu_T\cdot\nu_F) \operatorname{dist}(z_3,F)
   \varphi_1^2\varphi_2^2 (D\varphi_3\nu_F)
= 30 \varphi_1^2\varphi_2^2 .
\end{equation*}
Hence, $\zeta_F\in H^2(\Omega)$.

If $F\in\faces_\ell(\Omega)$, it holds that
$\zeta_F\in H^2_0(\omega_F)$.
Define the operator
$\widetilde{\mathcal{A}}:  V_\ell \to V$
which acts as
\begin{equation*}
\widetilde{\mathcal{A}} v_\ell
  := \mathcal A v_\ell 
  + \sum_{F\in\faces(\Omega\cup\Gamma_S\cup\Gamma_F)} 
 \left( \fint_F \frac{\partial (v_\ell-\mathcal A v_\ell)}{\partial\nu_F}\,ds
 \right) \zeta_F .   
\end{equation*}
An immediate consequence of this choice reads as
\begin{equation*}
 \fint_F \partial \widetilde{\mathcal{A}} v_\ell / \partial\nu_F\,ds  
 = \fint_F \partial v_\ell /\partial\nu_F\,ds 
\quad\text{for all } F\in\faces_\ell.
\end{equation*}
An integration by parts shows the integral mean property
of the Hessian $\Pi_\ell^0 D^2 \widetilde{\mathcal{A}} = D^2_\nc$.
The scaling $\norm{\zeta_F}_{L^2(T)}\lesssim h_T^2$
and the trace inequality \citep{CarstensenFunken2000,DiPietroErn2012}
prove, for any $T\in\tri_\ell$, that
\begin{align*}
&h_T^{-2} 
  \Big\Vert 
   \sum_{F\in\faces(T)}
     \left( \fint_F \frac{\partial (v_\ell-\mathcal A v_\ell)}{\partial\nu_F}\,ds
     \right) \zeta_F   \Big\Vert_{L^2(T)}
\\
& \qquad
\lesssim 
\sum_{F\in\faces(T)} \Big\vert
       \fint_F \frac{\partial (v_\ell-\mathcal A v_\ell)}{\partial\nu_F}\,ds
        \Big\vert
\\
& \qquad
\lesssim
h_T^{-1}\norm{D_\nc(v_\ell-\mathcal A v_\ell)}_{L^2(T)}  
   + \norm{D^2_\nc(v_\ell-\mathcal A v_\ell)}_{L^2(T)} .
\end{align*}
This together with the first step of the proof and
 inverse estimates \citep{BrennerScott2008} show that
\begin{equation}\label{e:TildeAestimate}
\begin{aligned}
\norm{h_\ell^{-2} (v_\ell- \widetilde{\mathcal{A}} v_\ell)}_{L^2(\Omega)}
&
+\norm{h_\ell^{-1}D_\nc (v_\ell- \widetilde{\mathcal{A}} v_\ell)}_{L^2(\Omega)}
+ \norm{D^2_\nc(v_\ell- \widetilde{\mathcal{A}} v_\ell)}_{L^2(\Omega)}
\\ 
& \qquad\qquad\qquad\qquad\qquad\qquad\qquad
\lesssim 
\min_{v\in V}\norm{D^2_\nc(v_\ell-v)}_{L^2(\Omega)}.
\end{aligned}
\end{equation}

{\em Step 3.}
On any triangle $T=\operatorname{conv}\{z_1,z_2,z_3\}$
with nodal basis functions 
$\varphi_1,\varphi_2,\varphi_3\in\mathcal P_0(T)$,
the volume bubble function is defined as
\begin{equation*}
\tilde{\bm\flat}_{T} := 2520\,\varphi_1^2\varphi_2^2\varphi_3^2
\in H^2_0(\operatorname{int}(T))
\end{equation*}
and satisfies $\fint_T\tilde{\bm\flat}_{T}\,dx = 1$.
Define
\begin{align*}
\mathcal{C}v_\ell
 :=\tilde{\mathcal{A}} v_\ell +\sum_{T\in\tri_\ell}
  \left(\fint_T(v_\ell-\tilde{\mathcal{A}} v_\ell)\,dx\right) \tilde{\bm\flat}_{T}.
\end{align*}
The difference $v_\ell-\mathcal{C}v_\ell$ is $L^2$ orthogonal to all
piecewise constant functions.
Since $\tilde{\bm\flat}_T$ vanishes on $F\in\faces_\ell$,
$\mathcal{C}$ enjoys the integral mean property
$\Pi_\ell^0 D^2 \mathcal{C}= D^2_\nc$.
The fact that
$\norm{\tilde{\bm\flat}_T}_{L^\infty(T)}\lesssim 1$
and the H\"older inequality prove
\begin{equation*}
 \left\| 
  \fint_T(v_\ell-\tilde{\mathcal{A}}v_\ell)\,dx\;\tilde{\bm\flat}_T\right\|_{L^2(T)}
\lesssim \norm{v_\ell-\tilde{\mathcal{A}} v_\ell}_{L^2(T)}.
\end{equation*}
Hence, the triangle inequality, 
\eqref{e:TildeAestimate} and inverse estimates prove the
claimed error estimate for $\mathcal C$.
\end{proof}

\begin{remark}
 The operator $\mathcal C$ maps into a discrete space, namely
 the sum of $V_{\mathrm{HCT}}(\tri_\ell)$ and 
  $\mathcal P_6(\tri_\ell)\cap V$.
\end{remark}

\begin{corollary}[discrete Poincar\'e-Friedrichs inequality for Morley functions]
 \label{c:dFMorley}
There exists a positive constant $C_{dF}$ such that
any $v_\ell \in V_\ell$ satisfies
\begin{equation*}
  \norm{v_\ell} 
\leq C_{dF} \operatorname{diam}(\Omega)^2 \ennormnc{v_\ell} .
\end{equation*}

\end{corollary}
\begin{proof}
 The proof follows from the triangle inequality
\begin{equation*}
 \norm{v_\ell}
 \leq
 \norm{v_\ell - \mathcal C v_\ell} + \norm{\mathcal C v_\ell} .
\end{equation*}
The first term on the right-hand side can be bounded via
\eqref{e:CompanionApproxStab} while the second term 
for $\mathcal C v_\ell\in V$ is controlled by a 
Poincare-Friedrichs-type estimate and the stability of the
operator $\mathcal C$.
\end{proof}

\subsection{$L^2$ Error Estimate for the Morley FEM}
\label{ss:Linear}            
This section presents $L^2$ and best-approximation error estimates
for the Morley finite element discretisation of the linear biharmonic
equation. The companion operator from 
Subsection~\ref{ss:companion}
allows the proof of an $L^2$ error estimate for possibly singular
solutions of the biharmonic equation.
Given $f\in L^2(\Omega)$,
the weak formulation seeks $u\in V$ such that
\begin{equation}\label{e:LinProbExact}
  a(u,v) =  b(f,v)
  \quad\text{for all } v\in V.
\end{equation}
Throughout this paper, $0<s\leq 1$ indicates the elliptic regularity
of the solution to \eqref{e:LinProbExact} in the sense that
$\norm{u}_{H^{2+s}(\Omega)}\leq C(s) \norm{f}_{L^2(\Omega)}$.

The Morley finite element discretisation of
\eqref{e:LinProbExact} seeks $u_\ell\in V_\ell$
such that
\begin{equation}\label{e:LinProbDiscr}
  a_\nc(u_\ell,v_\ell) = b(f, v_\ell)
  \quad\text{for all } v_\ell \in V_\ell.
\end{equation}
The following best-approximation is a refined version of
 a result of \citet{Gudi2010}. An alternative proof of the version
 stated here is given in \citep{LiGuanMao2014}.

\begin{proposition}[best-approximation result]\label{p:LinBestAppx}
 The exact solution $u$ of \eqref{e:LinProbExact} and the
 discrete solution $u_\ell$ of \eqref{e:LinProbDiscr}
 satisfy
 \begin{equation*}
  \ennormnc{u-u_\ell}
  \lesssim \norm{(1-\Pi_\ell^0) D^2 u}_{L^2(\Omega)}
           + \osc_2(f,\tri_\ell)  .
 \end{equation*}
\end{proposition}

\begin{proof}
 The projection property \eqref{e:CompanionProj} of the 
 interpolation operator $\Imorl_\ell$ and the Pythagoras
theorem show that
\begin{equation*}
 \ennormnc{u-u_\ell}^2
  = \ennormnc{u_\ell-\Imorl_\ell u}^2
   + \ennormnc{u- \Imorl_\ell u}^2.
\end{equation*}
Since $\ennormnc{u- \Imorl_\ell u}=\norm{(1-\Pi_\ell^0)D^2 u}$, it remains
to estimate the first term on the right-hand side.
Set $\varphi_\ell:=u_\ell-\Imorl_\ell u$. The properties of the
companion operator from Proposition~\ref{p:companion}
show that
\begin{equation*}
 \ennormnc{u_\ell-\Imorl_\ell u}^2
 = a_\nc(u_\ell-u,\varphi_\ell)
 = b(f,\varphi_\ell-\mathcal C \varphi_\ell) 
    + ((1-\Pi_\ell^0)D^2 u,D^2 _\nc(\mathcal C -1)\varphi_\ell)_{L^2(\Omega)}.
\end{equation*}
The approximation and stability properties 
\eqref{e:CompanionApproxStab} show that this is bounded
by
\begin{equation*}
 (\norm{h_\ell^2 f} + \norm{(1-\Pi_\ell^0)D^2 u}) \ennormnc{\varphi_\ell}.
\end{equation*}
The efficiency $\norm{h_\ell^2 f} 
\lesssim \norm{(1-\Pi_\ell^0)D^2 u} + \osc_2(f,\tri_\ell)$
follows from the arguments of \citet{VerfBook1996}, see, e.g.,
\citep[Prop.~3.1]{Gallistl2014}. This concludes the proof.
\end{proof}

Error estimates for the Morley FEM in the $L^2$ norm are well-established
\citep{LascauxLesaint1975} for the case of a smooth
solution $u\in V\cap H^3(\Omega)$. The smoothness enters the
classical proofs in that traces of certain second-order
derivatives are assumed to exist. This smoothness assumption
is satisfied for the purely clamped case $\partial\Omega=\Gamma_C$
where it is known \citep{BlumRannacher1980,MelzerRannacher1980} that
$u\in H^{5/2+\varepsilon}$ for some $\varepsilon>0$.
For the more general boundary conditions considered here,
this smoothness assumption is not satisfied in general.
The new companion operator $\mathcal C$ from 
Proposition~\ref{p:companion}
allows the proof of an $L^2$ error estimate for any $u\in V$.

\begin{proposition}[$L^2$ control for the linear problem]\label{p:LinL2}
The exact solution $u$ of \eqref{e:LinProbExact} and the
discrete solution $u_\ell$ of \eqref{e:LinProbDiscr} satisfy
\begin{align*}
 \norm{u-u_\ell}
 \lesssim
 \hnull^s\left(\ennormnc{u-u_\ell} + \osc_2(f,\tri_\ell)\right).
\end{align*}
\end{proposition}
\begin{proof}
 Let $e:=u-u_\ell $ and let $z\in V$ denote the solution of
 \begin{equation*}
    a(z,v) = b(e,v) \quad\text{for all }v\in V.
 \end{equation*}
Since $\Pi_\ell^0(u_\ell -\mathcal{C} u_\ell) = 0$
by Proposition~\ref{p:companion}, it holds that
\begin{equation}\label{e:duality1}
 \begin{aligned}
 \norm{e}^2
 &
  = b(\mathcal{C} u_\ell - u_\ell , e) + b(e,u-\mathcal{C} u_\ell) \\
 &
  = b(\mathcal{C} u_\ell - u_\ell ,(1-\Pi_\ell^0) e) 
     + a(z,u-\mathcal{C} u_\ell).
 \end{aligned}
\end{equation}
Piecewise Poincar\'e inequalities, the discrete Friedrichs
inequality \citep[Thm.~10.6.12]{BrennerScott2008},
 and \eqref{e:CompanionApproxStab} lead to
\begin{equation*}
   b(\mathcal{C} u_\ell  - u_\ell, (1-\Pi_\ell^0)e)
   \lesssim \hnull^3 \ennormnc{e}^2.
\end{equation*}
The second term of the right-hand side in \eqref{e:duality1}
satisfies
\begin{equation}\label{e:duality2}
 a(z,u-\mathcal{C} u_\ell) 
= a_\nc(z,u - u_\ell) + a_\nc(z,u_\ell - \mathcal{C} u_\ell).
\end{equation}
The projection property \eqref{e:MorleyInterpolProjProp}
of $\Imorl_\ell$, the problems 
\eqref{e:LinProbExact} and \eqref{e:LinProbDiscr},
the Cauchy inequality
and the approximation and stability properties 
\eqref{e:MorleyInterpolApprxStab} prove for the first
term of the right-hand side in \eqref{e:duality2} that
\begin{equation*}
a_\nc(z,u-u_\ell) = b(f,z-\Imorl_\ell z) 
   \lesssim \norm{h_\ell^2 f}_{L^2(\Omega)}
       \norm{ (1-\Pi_\ell^0) D^2 z}_{L^2(\Omega)} .
\end{equation*}
The integral mean property  \eqref{e:CompanionProj} 
of $\mathcal{C}$ and the approximation and stability
properties \eqref{e:CompanionApproxStab} prove
for the second term of \eqref{e:duality2} that
\begin{equation*}
 a_\nc(z,u_\ell - \mathcal{C} u_\ell)
= a_\nc(z - \Imorl_\ell z ,u_\ell - \mathcal{C} u_\ell)
\lesssim\ennormnc{u-u_\ell} \norm{(1-\Pi_\ell^0) D^2 z}_{L^2(\Omega)}.
\end{equation*}
The regularity estimates of 
\citep{BlumRannacher1980,Grisvard1985} and the stability
of the problem \eqref{e:LinProbExact} prove that
\begin{equation*}
 \norm{(1-\Pi_\ell^0) D^2 z}_{L^2(\Omega)}
\lesssim \hnull^s \norm{z}_{H^{2+s}(\Omega)}
\lesssim \hnull^s \norm{e}_{L^2(\Omega)}.
\end{equation*}
Efficiency estimates in the spirit of \citep{VerfBook1996} show that
\begin{equation*}
\norm{h_\ell^2 f}_{L^2(\Omega)}
 \lesssim \ennormnc{u-u_\ell} + \osc_2(f,\tri_\ell) .
\end{equation*}
The combination of the foregoing estimates concludes the proof.
\end{proof}

\section{Morley FEM for the Biharmonic Eigenvalue Problem}\label{s:evalest}

This section is devoted to the Morley finite element discretisation of
the biharmonic eigenvalue problem.
Subsection~\ref{ss:Cluster} describes an abstract framework for the
discretisation of selfadjoint eigenproblems. 
Subsection~\ref{ss:EVPdiscret} presents the finite element method along
with a new $L^2$ error estimate. Error estimates for the eigenfunctions
are given in Subsection~\ref{ss:nonstandardritz}--\ref{ss:evalest}.

\subsection{Abstract Approximation of Eigenvalue Clusters} \label{ss:Cluster}

Let $(H,a(\cdot,\cdot))$ be a separable Hilbert space over $\mathbb R$
 with induced norm
$\norm{\cdot}_a$ and let $b(\cdot,\cdot)$ be a scalar product on $H$
with induced norm $\norm{\cdot}_b$ such that the embedding
$(H,\norm{\cdot}_a)\hookrightarrow (H,\norm{\cdot}_b)$ is compact.
In the applications of this paper, $a$ and $b$ are the bilinear
forms defined in Subsection~\ref{ss:datastruct} and, hence,
no notational distinction is made for the possibly more general
bilinear forms $a$, $b$ in this subsection.
Consider the following eigenvalue problem: Find
eigenpairs $(\lambda,u)\in\mathbb R \times H$ with
$\norm{u}_b = 1$ such that
\begin{equation}\label{e:Cluster:AbstrEVP}
 a(u,v) = \lambda b(u,v)\quad\text{for all }v\in H.
\end{equation}
It is well known from the spectral theory of selfadjoint compact
operators
\citep{Chatelin1983,Kato1966}
that the eigenvalue problem \eqref{e:Cluster:AbstrEVP} has
countably many eigenvalues, which are real and positive
with $+\infty$ as only possible accumulation point.
Suppose that the eigenvalues are enumerated as
\begin{equation*}
 0<\lambda_1\leq\lambda_2\leq\lambda_3\leq\dots
\end{equation*}
and let $(u_1,u_2,u_3,\dots)$ be some $b$-orthonormal system of
corresponding eigenfunctions. For any $j\in\mathbb N$, the eigenspace
corresponding to $\lambda_j$ is defined as
\begin{equation*}
E(\lambda_j) 
:= \{ u\in H \mid (\lambda_j,u)
       \text{ satisfies } \eqref{e:Cluster:AbstrEVP}\}  
=
\operatorname{span}\{u_k \mid k\in\mathbb N \text{ and } \lambda_k =\lambda_j\}
.
\end{equation*}
In the present case of an eigenvalue problem of 
(the inverse of) a compact operator,
the spaces $E(\lambda_j)$ have finite dimension.
The discretisation of \eqref{e:Cluster:AbstrEVP} is based on a family
(over a countable index set $I$) of separable 
(not necessarily finite-dimensional) Hilbert spaces 
$H_\ell$  with scalar products
$a_\nc(\cdot,\cdot)$ and $b_\nc(\cdot,\cdot)$ on $H+H_\ell$ with
induced norms $\norm{\cdot}_{a,\nc}$ and $\norm{\cdot}_{b,\nc}$ such that
$a_\nc$ and $b_\nc$ coincide with $a$ and $b$ when restricted to
$H$
\begin{equation*}
 a_\nc|_{H\times H} = a \quad\text{and}\quad b_\nc|_{H\times H}=b.
\end{equation*}
The discrete eigenvalue problem seeks eigenpairs 
$(\lambda_\ell,u_\ell)\in\mathbb R \times H_\ell$ with
$\norm{u_\ell}_{b,\nc} = 1$ such that
\begin{equation}\label{e:Cluster:DiscrEVP}
 a_\nc(u_\ell,v_\ell) = \lambda_\ell b_\nc(u_\ell,v_\ell)
  \quad\text{for all }v_\ell\in H_\ell.
\end{equation}
The discrete eigenvalues can be enumerated
\begin{equation*}
 0<\lambda_{\ell,1}\leq\lambda_{\ell,2}\leq\lambda_{\ell,3}\dots
\end{equation*}
with corresponding $b_\nc$-orthonormal eigenfunctions
$(u_{\ell,1},u_{\ell,2},u_{\ell,3}\dots)$.
For a cluster of eigenvalues
$\lambda_{n+1},\dots,$ $\lambda_{n+N}$ of length
$N\in\mathbb{N}$, define the index set
$J:=\{n+1,\dots,n+N\}$ and the spaces
\begin{equation*}
W := \operatorname{span}\{u_j\mid j\in J\} \quad\text{and}\quad
W_\ell := \operatorname{span}\{u_{\ell,j}\mid j\in J\}.
\end{equation*}
The eigenspaces $E(\lambda_j)$ may differ for different $j\in J$.

Assume that the cluster is contained in a compact interval $[A,B]$ in the
sense that
\begin{equation*}
 \{\lambda_j\mid j\in J\}
 \cup
 \{\lambda_{\ell,j} \mid \ell\in I, j\in J\}
 \subseteq [A,B] .
\end{equation*}
This implies
\begin{equation}\label{e:Cluster:supleqAB}
\sup_{\ell\in I}\max_{(j,k)\in J^2}\max
  \Big\{\lambda_k^{-1}\lambda_{\ell,j},\lambda_{\ell,j}^{-1}\lambda_k\Big\}
  \leq B/A.
\end{equation}
Recall that $\operatorname{dim}(H_\ell) \in \mathbb N \cup \{\infty\}$
and let $J^C:=\{1,\dots,\dim(H_\ell)\} \setminus J$ denote the complement
of $J$.
Assume that the cluster is separated from the remaining part of
the spectrum in the sense that there exists a separation bound
\begin{equation*}
   M_J:= \sup_{\ell\in I}
          \sup_{j\in J^C}
          \max_{k\in J} \frac{\lambda_k}{\abs{\lambda_{\ell,j}-\lambda_k}}
        < \infty
 .
\end{equation*}
In particular, this assumption requires that the definition of the
cluster $J$ does not split a multiple eigenvalue.
Given $f\in H$, let $u\in H$ denote the unique solution to the
linear problem
\begin{equation*}
  a(u ,v) = b(f,v) \quad\text{for all } v \in H .
\end{equation*}
The quasi-Ritz projection $R_\ell u \in H_\ell$ is defined
as the unique solution to
\begin{equation*}
 a_\nc (R_\ell u , v_\ell) = b_\nc(f,v_\ell)
 \quad\text{for all } v_\ell \in H_\ell.
\end{equation*}
Let $P_\ell$ denote the $b_\nc$-orthogonal projection onto
$W_\ell$ and define
\begin{equation}\label{e:Cluster:defLambda}
 \Lambda_\ell  := P_\ell \circ R_\ell.
\end{equation}
For any eigenfunction $u\in W$, the function
$\Lambda_\ell u \in W_\ell$ is regarded as its approximation.
This approximation does not depend on the 
basis of $W_\ell$. Notice that $\Lambda_\ell u$ is neither computable
without knowledge of $u$ nor necessarily an eigenfunction.

The following result is essentially contained in the book
of \citet{StrangFix1973} and in \citep{CarstensenGedicke2011}
for a conforming finite element
discretisations.
The version stated here is proven in
\citep{Gallistl2014nc}.

\begin{proposition}\label{p:Cluster:GlLambdalDiff}
Any eigenpair $(\lambda,u)\in \mathbb R\times W$
of \eqref{e:Cluster:AbstrEVP}
with $\norm{u}_b = 1$ satisfies
\begin{equation*}
\begin{aligned}
  \norm{R_\ell u-\Lambda_\ell u}_{b,\nc}
&  \leq M_J \norm{u-R_\ell u}_{b,\nc} \quad\text{and}
\\
  \norm{u - P_\ell u}_{b,\nc}
   \leq
  \norm{u - \Lambda_\ell u}_{b,\nc} 
&  \leq (1+M_J) \norm{u-R_\ell u}_{b,\nc}.
\end{aligned}
\end{equation*}
\end{proposition}
\begin{proof}
 See \citep{Gallistl2014nc}.
\end{proof}

The following algebraic identity applies frequently
in the analysis. It states the important property that,
although $\Lambda_\ell u$ is no eigenfunction in general,
$\Lambda_\ell u$ satisfies an equation that is similar to an
eigenfunction property.

\begin{lemma}\label{l:LambdaPLemma}
 Any eigenpair
 $(\lambda, u)\in \mathbb{R}\times H$
 of \eqref{e:Cluster:AbstrEVP}
 satisfies
 \begin{equation*}
  a_\nc(\Lambda_\ell u,v_\ell) 
   = \lambda b_\nc(P_\ell u,v_\ell)
  \quad\text{for all } v_\ell \in H_\ell
   .
 \end{equation*}
 In other words, $R_\ell$ and $P_\ell$ commute,
 $P_\ell\circ R_\ell = R_\ell\circ P_\ell$.
 \end{lemma}
\begin{proof}
The proof is given in \citep[Lemma~2.2]{Gallistl2014}.
\end{proof}

The following theorem of \citet{KnyazevOsborn2006} gives an abstract
eigenvalue error estimate in case $H_\ell \subseteq H$.

\begin{theorem}[Corollary~3.4 of \citep{KnyazevOsborn2006}]
\label{t:Knyazev}
 Suppose $H_\ell \subseteq H$ and let, for $p\in\mathbb N$,
 $\lambda_p$ be an eigenvalue of \eqref{e:Cluster:AbstrEVP}
 with multiplicity $q\in\mathbb N$, so that
\begin{equation*}
 \lambda_{p-1}<\lambda_p=\dots=\lambda_{p+q-1}<\lambda_{p+q} 
\end{equation*}
(with the convention $\lambda_0:=0$) and suppose that
\begin{equation*}
 \min_{j={1,\dots,p-1}} \abs{\lambda_{\ell,j} - \lambda_p} \neq 0.
\end{equation*}
Let $T:H\to H$ denote the solution operator of the associated linear
problem, i.e., for given $f\in H$, $Tf\in H$ solves
\begin{equation*}
  a(Tf,v) = b(f,v) \quad\text{for all } v\in H.
\end{equation*}
Then, for any $k\in\{p,\dots,p+q-1\}$, the following estimate holds
\begin{equation*}
\begin{aligned}
&
\frac{\lambda_{\ell,k} - \lambda_p}{\lambda_{\ell,k}}
\\
&\quad
\leq
\bigg(
  1+  \max_{j={1,\dots,p-1}}
     \frac{\lambda_{\ell,j}^2\lambda_p^2}{\abs{\lambda_{\ell,j}-\lambda_p}^2}
     \sup_{\substack{f\in \operatorname{span}\{u_{\ell,1},\dots ,u_{\ell,p-1}\}\\ \norm{f}_a = 1 }} \norm{(1-R_\ell) T f}_a^2
\bigg) 
\sup_{\substack{u\in E(\lambda_p) \\ \norm{u}_a=1}}
           \inf_{v_\ell\in H_\ell} \norm{u-v_\ell}_a^2  
\end{aligned}
\end{equation*}
 where the maximum and supremum in the parentheses are
 $0$ for $p=1$.
\endproof
\end{theorem}

\begin{remark}
In this paper, the first supremum will usually be estimated through
(a power of)
some Friedrichs-type constant although it can be seen that in case
of a finite element space $V_\ell$ this quantity even decays as a
certain power of the maximum mesh-size.
\end{remark}

\begin{remark}\label{r:Cluster:KOdiminfOK}
In \citep{KnyazevOsborn2006} the result of Theorem~\ref{t:Knyazev} 
is stated for a finite-dimensional space $H_\ell$, but it
is valid even if $H_\ell$ has
infinite dimension. Only the finite dimension of the eigenspaces
is required. One way to see this is to trace carefully the arguments
in the proof of \citet{KnyazevOsborn2006}. For the reader's convenience,
another argument is given here that reduces the stated result for
$\dim H_\ell = \infty$ to the finite-dimensional case.
To this end, consider the finite-dimensional subspace
\begin{equation*}
\widetilde H_\ell := \operatorname{span}
\{ u_{\ell,1},\dots,u_{\ell,p+q-1},R_\ell u_p,\dots,R_\ell u_{p+q-1}, 
   R_\ell T u_{\ell,p},\dots R_\ell T {u_{\ell,{p-1}}}
\}
\subseteq H_\ell .
\end{equation*}
The finite-dimensional space $\widetilde H_\ell$ is constructed in such a way that the
first $p+q-1$ eigenvalues $\lambda_{\ell,1},\dots,\lambda_{\ell,p+q-1}$
that are relevant for the statement of Theorem~\ref{t:Knyazev}
are attained in $\widetilde H_\ell$ and similarly all further quantities
in the estimate are attained in this finite-dimensional space.
For instance,
\begin{equation*}
\sup_{\substack{u\in E(\lambda_p) \\ \norm{u}_a=1}}
           \inf_{v_\ell\in H_\ell} \norm{u-v_\ell}_a^2 
=
\sup_{\substack{u\in E(\lambda_p) \\ \norm{u}_a=1}}
           \norm{u-R_\ell u}_a^2 
=
\sup_{\substack{u\in \operatorname{span}\{u_p,\dots,u_{p+q-1}\} \\ \norm{u}_a=1}}
           \norm{u-R_\ell u}_a^2 
\end{equation*}
is realised in $\widetilde H$.
Theorem~\ref{t:Knyazev} can be employed for
$\widetilde H_\ell$ in its original version and is thereby also valid
for $H_\ell$ because the claimed inequality is the same.
\end{remark}

\begin{remark}
 The conformity assumption $H_\ell\subseteq H$ is essential for 
 the proof of Theorem~\ref{t:Knyazev} and the result may be
 not true in general for nonconforming approximations where
 $H_\ell\not\subseteq H$. Subsection~\ref{ss:evalest} will apply
 Theorem~\ref{t:Knyazev} to a modified setting.
\end{remark}

\begin{remark}
In Subsection~\ref{ss:evalest} below, Theorem~\ref{t:Knyazev}
will be applied to the case that
$ H_\ell := V\subseteq \widehat V_\ell :=V+V_\ell=:H$ 
where $V_\ell$ is a nonconforming finite element space and
$V$ itself is a subspace of the enhanced space $\widehat V_\ell$.
\end{remark}

\begin{remark}[normalisation]\label{r:normalis}
The eigenvalue problems in this paper are based on the normalisation
$\norm{\cdot}_{b,\nc} = 1$ and typically approximation quantities like
\begin{equation*}
\sup_{\substack{w\in W\\ \norm{w}_{b,\nc} = 1}}
\inf_{v_\ell\in W_\ell}         \norm{w-v_\ell}_{a,\nc}^2
\end{equation*}
arise in the analysis.
To see that this quantity essentially describes the angle
$\sin_{a,\nc}^2\angle(W,W_\ell)$ up to some scaling, consider the expansion
of $w$
in terms of the eigenfunctions of $W$. Then the eigenvalue problem
implies
\begin{equation*}
\begin{aligned}
\sin_{a,\nc}^2\angle(W,W_\ell)
&=
\sup_{w\in W\setminus\{0\}}
\frac{\inf_{v_\ell\in W_\ell}\norm{w-v_\ell}_{a,\nc}^2}{\norm{w}_{b,\nc}^2}
\frac{\norm{w}_{b,\nc}^2}{\norm{w}_{a,\nc}^2}
\\
&
\leq 
\frac{1}{\lambda_{n+1}}
\sup_{\substack{w\in W\\ \norm{w}_{b,\nc} = 1}}
\inf_{v_\ell\in W_\ell}         \norm{w-v_\ell}_{a,\nc}^2
\\
&
\leq 
\frac{\lambda_{n+N}}{\lambda_{n+1}} \sin_{a,\nc}^2\angle(W,W_\ell)
\leq \frac{B}{A}\sin_{a,\nc}^2\angle(W,W_\ell) .
\end{aligned}
\end{equation*}
This means that the error quantities are comparable up to a factor
described by the ratio of the cluster bounds.
\end{remark}

\subsection{Morley FEM Discretisation for the Eigenvalue Problem}
\label{ss:EVPdiscret}

The weak form of the biharmonic eigenvalue problem seeks
eigenpairs $(\lambda,u)\in \mathbb R\times V$ with $\norm{u}=1$
such that
\begin{equation}\label{e:ExactBihEVP}
 a(u,v) = \lambda b(u,v)\quad\text{for all } v\in V.
\end{equation}

The Morley finite element discretisation of
problem \eqref{e:ExactBihEVP} seeks 
$(\lambda_\ell,u_\ell)\in\mathbb R\times V_\ell$ 
with $\norm{u_\ell} = 1$ such that
\begin{equation}\label{e:DiscrBihEVP}
 a_\nc(u_\ell,v_\ell) = \lambda_\ell b(u_\ell,v_\ell)
 \quad\text{for all } v_\ell\in V_\ell.
\end{equation}
Recall the notation from Subsection~\ref{ss:Cluster} for
$H= V$ and $H_\ell = V_\ell$ and
the exact and discrete eigenvalues
\begin{equation*}
 0<\lambda_1\leq\lambda_2\leq\dots
 \quad\text{and}\quad
 0<\lambda_{\ell,1}\leq\dots\leq\lambda_{\ell,\dim(V_\ell)}
\end{equation*}
and their corresponding $b$-orthonormal systems of eigenfunctions
\begin{equation*}
(u_1,u_2,u_3,\dots)
 \quad\text{and}\quad 
(u_{\ell,1},u_{\ell,2},\dots,u_{\ell,\dim(V_\ell)}).
\end{equation*}
The eigenvalue cluster is described by the index set
$J:=\{n+1,\dots,n+N\}$ and the spaces
$W := \operatorname{span}\{u_j\mid j\in J\}$ and
$W_\ell := \operatorname{span}\{u_{\ell,j}\mid j\in J\}$.
The cluster is contained in the interval $[A,B]$.
Furthermore, the following separation condition is assumed
(cf.\ Subsection~\ref{ss:Cluster}).
\begin{equation}
  \label{e:separationMorley}
   M_J:= \sup_{\ell\in I}
          \sup_{j\in J^C}
          \max_{k\in J} \frac{\lambda_k}{\abs{\lambda_{\ell,j}-\lambda_k}}
        < \infty
 .
\end{equation}

\begin{proposition}[$L^2$ control]\label{p:EVPL2control}
 Provided $\hnull\ll 1$, any eigenpair 
$(\lambda,u)\in\mathbb R\times W$ of \eqref{e:ExactBihEVP}
with $\norm{u}=1$ satisfies for some constant $C_{L^2}$ that
\begin{equation*}
 \norm{u-P_\ell u}
\leq
 \norm{u-\Lambda_\ell u}
 \leq C_{L^2}(1+M_J)\hnull^s\ennormnc{u-\Lambda_\ell u} .
\end{equation*}
\end{proposition}
\begin{proof}
 The combination of Proposition~\ref{p:Cluster:GlLambdalDiff} with
 Proposition~\ref{p:LinL2} and Proposition~\ref{p:LinBestAppx}
 leads to
\begin{equation*}
  \norm{u-\Lambda_\ell u}
 \lesssim (1+M_J)\hnull^s
    (\ennormnc{u-\Lambda_\ell u} + \osc_2(\lambda u,\tri_\ell)).
\end{equation*}
Provided $\hnull\ll1$, the oscillation term can be absorbed.
\end{proof}

The following proposition is based on the comparison result from
Proposition~\ref{p:LinBestAppx} and states a best-approximation
property for $\Lambda_\ell u$.

\begin{proposition}[best-approximation result]\label{p:EVPcomparison}
Provided  $\hnull\ll 1$,
any eigenfunction $u\in W$ of \eqref{e:ExactBihEVP} with
$\norm{u}=1$ satisfies
\begin{equation*}
 \ennormnc{u-\Lambda_\ell u}
 \lesssim \norm{(1-\Pi_\ell^0) D^2u}_{L^2(\Omega)}.
\end{equation*}
\end{proposition}
\begin{proof}
Recall that the quasi-Ritz projection
$R_\ell u$ solves \eqref{e:LinProbDiscr} with
right-hand side $f=\lambda u$.
The triangle inequality proves
\begin{equation*}
 \ennormnc{u-\Lambda_\ell u}
\leq
\ennormnc{u-R_\ell u} + \ennormnc{R_\ell u - \Lambda_\ell u} .
\end{equation*}
Set $\varphi_\ell :=R_\ell u - \Lambda_\ell u$. The definition
of $R_\ell$ and the discrete problem
(cf.\ Lemma~\ref{l:LambdaPLemma}) prove that
\begin{equation*}
 \ennormnc{R_\ell u - \Lambda_\ell u}^2
=
a_\nc(R_\ell u - \Lambda_\ell u,\varphi_\ell)
=
\lambda b(u-P_\ell u, \varphi_\ell) .
\end{equation*}
Hence, the Cauchy and discrete Friedrichs inequalities
(Corollary~\ref{c:dFMorley}) and the $L^2$
control from Proposition~\ref{p:EVPL2control} prove that
\begin{equation*}
 \ennormnc{R_\ell u - \Lambda_\ell u}
\lesssim
\lambda (1+M_J) \hnull^s
\ennormnc{u-\Lambda_\ell u}.
\end{equation*}
The combination of the foregoing estimates with 
Proposition~\ref{p:LinBestAppx} results in
\begin{equation*}
  \ennormnc{u-\Lambda_\ell u}
\lesssim
\norm{(1-\Pi_\ell^0) D^2 u}_{L^2(\Omega)}
+
\lambda (1+M_J) \hnull^s \ennormnc{u-\Lambda_\ell u}
+
\osc_2 (\lambda u,\tri_\ell) .
\end{equation*}
If $\hnull\ll 1$ is sufficiently small, the higher-order terms
on the right-hand side can be absorbed.
\end{proof}

\subsection{A Nonstandard Quasi-Ritz Projection}
\label{ss:nonstandardritz}

This subsection introduces the setting which is necessary for the 
eigenvalue estimates of Subsection~\ref{ss:evalest}.

Define $\widehat V_\ell := V + V_\ell$ as the sum of the continuous
and the discrete space.
Given $f\in V$, let $u\in V$ denote the solution to
\eqref{e:LinProbExact}, namely
\begin{equation*}
 a(u,v) = b(f,v) \quad\text{for all }v\in V .
\end{equation*}
The quasi-Ritz projection 
$\widehat R_\ell u\in\widehat V_\ell$ is defined as the solution of
\begin{equation*}
 a_\nc(\widehat R_\ell u, \hat v_\ell ) = b(f,\hat v_\ell)
 \quad\text{for all } \hat v_\ell \in \widehat V_\ell. 
\end{equation*}

\begin{remark}
 This definition corresponds to the definition
 of $R_\ell$ of Subsection~\ref{ss:Cluster} with $H_\ell$ replaced
 by $\widehat V_\ell$. 
 It should be 
 emphasised that in the present case there is an inclusion
 $V\subseteq \widehat V_\ell$. This is an admissible choice in the
 framework of Subsection~\ref{ss:Cluster}.
\end{remark}

This setting leads to a new view on
nonconforming finite element schemes in the following sense:
Both $V$ and $V_\ell$ are subspaces of the space $\widehat V_\ell$
and the solutions $u\in V$ and $u_\ell \in V_\ell$ of 
\eqref{e:LinProbExact} and \eqref{e:LinProbDiscr}
are ``conforming approximations'' of 
$\widehat R_\ell u$.%
To the best of the author's know\-ledge, this is a new approach to
nonconforming finite elements
that has not been studied in the existing literature.

It is crucial that the nonconforming interpolation operator
$\Imorl_\ell$ is defined on $\widehat V_\ell$ as well as
$\widehat V_{\ell+m} = V + V_{\ell+m}$ with respect to a refined
triangulation $\tri_{\ell+m}$.
This operator and the conforming companion operator $\mathcal C $ from
Proposition~\ref{p:companion} establish suitable connections
between the spaces 
$V$, $V_\ell$, $\widehat V_\ell$, $V_{\ell+m}$ and $\widehat V_{\ell+m}$.
Those two operators displayed in Figure~\ref{f:evalestdiagram}
are the core of the analysis of $\widehat R_\ell$ 
which is essential to derive eigenvalue error estimates.
\begin{figure}
\centering
\begin{tikzpicture}[node distance=2.8cm, auto]
\node (A) {$V_\ell$};
\node (Z)[left of=A]{$V_{\ell+m}$};
\node (Y)[below of=Z]{$\widehat V_{\ell+m} := V + V_{\ell+m}$};
\node(B)[right of=A] {$ V $};
\node (C) [below of=A] {$\widehat V_\ell := V+V_\ell$};
\draw[->] (Z) -- (A) node[above,midway] {\footnotesize$\Imorl_\ell$};
\draw[transform canvas={yshift=0.5ex},->] (A) --(B) node[above,midway] {\footnotesize $\mathcal C $};
\draw[transform canvas={yshift=-0.5ex},->](B) -- (A) node[below,midway] {\footnotesize$\Imorl_\ell$}; 
\draw[transform canvas={xshift=-0.5ex},->] (C)--(A) node[left,midway]{\footnotesize$\Imorl_\ell$};
\draw[transform canvas={xshift=0.5ex},right hook->] (A)--(C) node[right,midway]{\footnotesize$\iota$};
\draw[right hook->](B)--(C) node[left,midway]  {\footnotesize$\iota$};
\draw[transform canvas={xshift=-0.5ex},->] (Y)--(Z) node[left,midway]{\footnotesize$\Imorl_{\ell+m}$};
\draw[transform canvas={xshift=0.5ex},right hook->] (Z)--(Y) node[right,midway]{\footnotesize$\iota$};
\draw[->] (Y) -- (A) node[left,midway] {\footnotesize$\Imorl_\ell$};
\end{tikzpicture}
 \caption[Mappings between the spaces $\widehat V_\ell$, 
          $\widehat V_{\ell+m}$, $V$, $V_\ell$
          and $V_{\ell+m}$]
         {Mappings between the spaces $\widehat V_\ell$, 
          $\widehat V_{\ell+m}$, $V$, $V_\ell$
          and $V_{\ell+m}$;
          $\iota$ is the inclusion. \label{f:evalestdiagram}}
\end{figure}
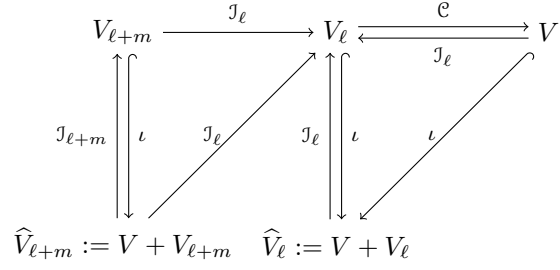

The following proposition gives an $L^2$ error estimate for the
quasi-Ritz projection $\widehat R_\ell$.

\begin{proposition}[$L^2$ error estimate for $\widehat R_\ell$]
\label{p:L2pesudoLinPoisson}
Let $u\in V$ solve the linear problem
\eqref{e:LinProbExact} with right-hand
side $f\in V$. Then, $\widehat R_\ell u$ satisfies the following
$L^2$ error estimate
\begin{equation*}
 \norm{u - \widehat R_\ell u}
\lesssim
\hnull^s
\ennormnc{u - \widehat R_\ell u}.
\end{equation*}
\end{proposition}

\begin{remark}
 The conformity $V\subseteq \widehat V_\ell$ shows that $u$ is
 the $a_\nc$-orthogonal projection of $\widehat R_\ell u$ onto
 $V$. Therefore, one may think of using a standard duality argument
 for the proof of the $L^2$ error control. 
 Indeed, this procedure can be
 applied, but it will not immediately lead to a right-hand side that is 
 explicit in the mesh-size $\hnull$.
 Therefore, the proof of Proposition~\ref{p:L2pesudoLinPoisson}
 employs a different technique based on the operators
 $\Imorl_\ell$ and $\mathcal C$ to obtain an estimate in terms of 
 $\hnull$.
\end{remark}

\begin{proof}[Proof of Proposition~\ref{p:L2pesudoLinPoisson}]
 Set $\hat e:= u - \widehat R_\ell u$ and 
 let $z\in V$ denote the solution to
\begin{equation*}
 a(z,w) = b(\hat e,w) \quad\text{for all } w\in V.
\end{equation*}
With the companion operator $\mathcal C$ from 
Proposition~\ref{p:companion} and the 
interpolation operator $\Imorl_\ell$, it follows that
\begin{equation}\label{e:rhatL2a}
 \norm{\hat e}^2
=
b((1-\mathcal C )\Imorl_\ell \hat e,\hat e)
+b((1-\Imorl_\ell)\hat e,\hat e)
+b(\mathcal C  \Imorl_\ell \hat e, \hat e) .
\end{equation}
The Cauchy inequality and the
error estimates \eqref{e:MorleyInterpolApprxStab} and
\eqref{e:CompanionApproxStab} bound the first two
terms on the right-hand side as
\begin{equation*}
 b((1-\mathcal C )\Imorl_\ell \hat e,\hat e)
+b((1-\Imorl_\ell)\hat e,\hat e)
\lesssim
\hnull^2 \ennormnc{\hat e} \norm{\hat e} .
\end{equation*}
Since $a(z,\hat e) =a(\hat e,z)=a(u-\widehat R_\ell u, z)=0$
 by the definition of $\widehat R_\ell$,
the remaining term of \eqref{e:rhatL2a} satisfies
\begin{equation*}
\begin{aligned}
 b(\mathcal C  \Imorl_\ell \hat e, \hat e)
&
= 
a(z,\mathcal C  \Imorl_\ell \hat e)
\\
&
=
a_\nc(z,(\Imorl_\ell-1) \hat e)
+
a_\nc(z,(\mathcal C -1) \Imorl_\ell \hat e).
\end{aligned}
\end{equation*}
The projection properties
\eqref{e:MorleyInterpolProjProp} and
\eqref{e:CompanionProj}
imply that
$D^2_\nc(\Imorl_\ell-1) \hat e$ as well as 
$D^2_\nc(\mathcal C -1) \Imorl_\ell \hat e$
are $L^2$-orthogonal onto piecewise constants.
This and the elliptic regularity 
show that
\begin{equation*}
\begin{aligned}
 &a_\nc(z,(\Imorl_\ell-1) \hat e)
+
a_\nc(z,(\mathcal C -1) \Imorl_\ell \hat e)
\\
&\quad
=
 ((1-\Pi_\ell^0) D^2 z,D^2_\nc(\Imorl_\ell-1) \hat e)_{L^2(\Omega)}
+
 ((1-\Pi_\ell^0) D^2 z,D^2_\nc(\mathcal C -1) \Imorl_\ell \hat e)_{L^2(\Omega)}
\\
&\quad
\lesssim
\hnull^s \norm{z}_{H^{2+s}(\Omega)}\ennormnc{\hat e}
\lesssim
\hnull^s \norm{\hat e}\ennormnc{\hat e} .
\end{aligned}
\end{equation*}
The combination of the above estimates concludes the proof.
\end{proof}

The next proposition states that the error
$u-\widehat R_\ell u$ in the energy norm is comparable
with the best-approximation of $Du$ by piecewise constants.

\begin{proposition}[comparison for $\widehat R_\ell$]
                    \label{p:PseudoBA}
Let $u\in V$ solve \eqref{e:LinProbExact} with right-hand
side $f\in V$. Then the quasi-Ritz projection
$\widehat R_\ell u$ satisfies
\begin{equation*}
 \ennormnc{u-\widehat R_\ell u}
\lesssim
\norm{(1-\Pi_\ell^0) D^2 u}_{L^2(\Omega)}
+
\osc_2 (f,\tri_\ell).
\end{equation*}
\end{proposition}
\begin{proof}
 The triangle inequality shows for the
 nonconforming interpolation operator $\Imorl_\ell$ that
\begin{equation*}
 \ennormnc{u-\widehat R_\ell u}
  \leq \ennormnc{\widehat R_\ell u-\Imorl_\ell u}
   + \ennormnc{u- \Imorl_\ell u}.
\end{equation*}
Since $\ennormnc{u- \Imorl_\ell u}=\norm{(1-\Pi_\ell^0)D^2 u}$
by the projection property \eqref{e:MorleyInterpolProjProp},
it remains to estimate the first term on the right-hand side.
Set $\hat \varphi_\ell:=\widehat R_\ell u-\Imorl_\ell u$.
The definition of $\widehat R_\ell$, the projection property
\eqref{e:MorleyInterpolProjProp}
and the properties of the companion operator from 
Proposition~\ref{p:companion} yield
\begin{equation*}
\begin{aligned}
 \ennormnc{\widehat R_\ell u-\Imorl_\ell u}^2
 &
 = a_\nc(\widehat R_\ell u-\Imorl_\ell u,\hat \varphi_\ell)
 \\
 &
 = b(f,\hat \varphi_\ell) - a_\nc(u, \Imorl_\ell \hat \varphi_\ell)
 \\
 &
 = b(f,\hat \varphi_\ell-\mathcal C \Imorl_\ell\hat \varphi_\ell) 
    - a_\nc(u,  (1-\mathcal C ) \Imorl_\ell\hat \varphi_\ell) .
\end{aligned}
\end{equation*}
The triangle inequality and the
approximation and stability properties 
\eqref{e:MorleyInterpolApprxStab} and 
\eqref{e:CompanionApproxStab} show for the first term that
\begin{equation*}
b(f,\hat \varphi_\ell-\mathcal C \Imorl_\ell\hat \varphi_\ell)
\lesssim
 \norm{h_\ell^2 f} \ennormnc{\hat \varphi_\ell}.
\end{equation*}
The known efficiency %
\begin{equation*}
 \norm{h_\ell^2 f} 
 \lesssim \norm{(1-\Pi_\ell^0)D^2 u} + \osc_2(f,\tri_\ell)
\end{equation*}
follows from the arguments of \citet{VerfBook1996}.

The projection property \eqref{e:CompanionProj} of 
$\mathcal C $ and \eqref{e:CompanionApproxStab} reveal
\begin{equation*}
 a_\nc(u, (1-\mathcal C ) \Imorl_\ell\hat \varphi_\ell)
=
 ((1-\Pi_\ell^0) D^2 u, D^2 _\nc (1-\mathcal C ) \Imorl_\ell\hat \varphi_\ell)_{L^2(\Omega)}.
\end{equation*}
This and the stability properties 
\eqref{e:MorleyInterpolApprxStab} and 
\eqref{e:CompanionApproxStab} conclude the proof.
\end{proof}

\subsection{Eigenvalue Error Estimates}\label{ss:evalest}

This section extends the results of the foregoing 
subsection to eigenvalue problems.
This leads to eigenvalue error estimates for the Morley finite
element method.

Note that $\widehat V_\ell$ equipped with the scalar product
$a_\nc$ is a Hilbert space.
The space $\widehat V_\ell$ is a subspace of the finite product
$H^2(\tri_\ell):=\prod_{T\in\tri_\ell} H^2(\operatorname{int}(T))$
and  the embedding
$(\widehat V_\ell, \ennormnc{\cdot})\to (L^2(\Omega), \norm{\cdot})$
is compact for a fixed triangulation $\tri_\ell$
(for more details on such broken Sobolev spaces see \citep{BuffaOrtner2009}).
Hence, the eigenvalue problem
\begin{equation}\label{e:extEVPPoiss}
 a_\nc(\hat u_\ell, \hat v_\ell) 
= \hat \lambda_\ell b(\hat u_\ell, \hat v_\ell)
 \quad\text{for all } \hat v_\ell\in\widehat V_\ell
\end{equation}
has a countable and discrete spectrum
\begin{equation*}
 0<\hat\lambda_{\ell,1}\leq\hat\lambda_{\ell,2}\leq\cdots
\end{equation*}
with corresponding $b$-orthonormal eigenfunctions
$(\hat u_{\ell,1},\hat u_{\ell,2},\dots)$.
For an eigenvalue cluster described by the index set 
$J=\{n+1,\dots,n+N\}$,
the set 
$\widehat W_\ell := \operatorname{span}\{\hat u_{\ell,j}\mid j\in J\}$
describes the corresponding invariant subspace with
the $L^2$ projection $\widehat P_\ell$ onto $\widehat W_\ell$ and let
$\widehat\Lambda_\ell := \widehat P_\ell \circ \widehat R_\ell$.

The eigenvalue problem \eqref{e:extEVPPoiss} is related
to the (inverse of) a compact operator for each triangulation
$\tri_\ell$. 
The first important observation is that the spectrum is robust
under mesh-refinement.
\begin{proposition}\label{p:extSpecRobustPois}
Let $(\tri_\ell)_{\ell\in\mathbb N_0}$ be a sequence of
nested triangulations with $\hnull\ll 1$.
Then any $j\in\mathbb N$ and the constant $C$ from
the estimate in \eqref{e:MorleyInterpolApprxStab} satisfy

\begin{equation}\label{e:twosided}
  \frac{\lambda_{\ell,j}}{1+ C\norm{h_\ell}_\infty^4 \lambda_{\ell,j}}
 \leq
 \hat\lambda_{\ell,j}
 \leq \lambda_{\ell,j} . 
\end{equation}
In particular, if $\norm{h_\ell}_\infty\to 0$ as $\ell\to\infty$,
one has convergence
 $\hat\lambda_{\ell,j} \to \lambda_j$.
\end{proposition}

\begin{proof}
 The min-max principle \citep{WeinsteinStenger1972}
 shows, for any $j\in\mathbb N$, that
\begin{equation*}
 \hat\lambda_{\ell,j} \leq \min \{\lambda_j,\lambda_{\ell,j}\} .
\end{equation*}
An application of the methodology of
\citep[Thms.~1--2]{CarstensenGallistl2014} yields
the lower eigenvalue bound in case that
$\norm{h_\ell}_\infty$ is sufficiently small
\begin{equation*}
 \frac{\lambda_{\ell,j}}{1+ C\norm{h_\ell}_\infty^4 \lambda_{\ell,j}}
 \leq
 \hat\lambda_{\ell,j}
\end{equation*}
for some constant $C\approx 1$. 
In fact, the arguments from \citep{CarstensenGallistl2014} can be
applied in this modified setting because the Morley interpolation
operator $\Imorl_{\ell}$ is defined for functions in $\widehat V_\ell$
and satisfies the projection property \eqref{e:MorleyInterpolProjProp}
and the approximation and stability property
 \eqref{e:MorleyInterpolApprxStab}.

Altogether one has the two-sided estimate
\eqref{e:twosided}.
This implies the convergence
$\abs{\lambda_{\ell,j} - \hat\lambda_{\ell,j}}\to 0$ as $\ell\to\infty$.
The triangle inequality and the a~priori estimates
of \citep{CarstensenGallistl2014}
prove $\hat\lambda_{\ell,j} \to \lambda_j$.
\end{proof}

The robustness implies the following separation bound.

\begin{corollary} \label{c:separationMhat}
 Provided $\hnull\ll 1$,
 there exists a separation constant for the cluster $J$ in the
sense that
\begin{equation}\label{e:separationHat}
 \widehat M_J:= \sup_{\tri_\ell\in \mathbb T}
    \max_{j\in \mathbb N \setminus J}
    \max_{k\in J}
    \max\left\{
     \frac{\hat \lambda_{k,\ell}}{\abs{\lambda_j-\hat\lambda_{k,\ell}}} ,
     \frac{\hat \lambda_{k,\ell}}{\abs{\lambda_{\ell,j}-\hat\lambda_{k,\ell}}},
     \frac{\lambda_k}{\abs{\hat\lambda_{\ell,j}-\lambda_k}},
     \frac{\lambda_k}{\abs{\lambda_{\ell,j}-\lambda_k}}
        \right\}       
        < \infty .
\end{equation}
This formula uses the convention
$\lambda_{\ell,j} := \lambda_{\ell,\dim (V_\ell)}$
for $j>\dim (V_\ell)$.
\end{corollary}

\begin{remark}
 The separation condition \eqref{e:separationHat} implies
 \eqref{e:separationMorley}  with $M_J\leq \widehat M_J$.
\end{remark}

This separation constant allows the use of the framework of
Subsection~\ref{ss:Cluster} where the space $V$ is
approximated by $\widehat V_\ell$.

\begin{proposition}[$L^2$ error estimate for $\widehat \Lambda_\ell$]
                   \label{p:L2pesudoevpPoiss}
 Provided $\hnull\ll 1$, 
 any eigenpair $(\lambda,u)\in \mathbb R \times W$
 of \eqref{e:ExactBihEVP} with $\norm{u}=1$ satisfies
\begin{equation*}
\norm{u-\Lambda_\ell u}
+
\norm{u-\widehat \Lambda_\ell u}
\lesssim
(1+\widehat M_J) \hnull^s
\norm{(1-\Pi_\ell^0) D^2 u} .
\end{equation*}
\end{proposition}

\begin{proof}
 An immediate consequence of Proposition~\ref{p:Cluster:GlLambdalDiff}
 (where $H_\ell$ is replaced by $\widehat V_\ell$
  and $\Lambda_\ell$ is replaced  by $\widehat \Lambda_\ell$)
and Proposition~\ref{p:PseudoBA} 
reads
\begin{equation*}
 \norm{u-\widehat\Lambda_\ell u}
 \leq 
 (1+\widehat M_J)
 \norm{u-\widehat R_\ell u}
 \lesssim
 (1+\widehat M_J)
 \hnull^s(
 \norm{(1-\Pi_\ell^0) D^2 u}_{L^2(\Omega)}
+
\osc_2(\lambda u,\tri_\ell) 
 ).
\end{equation*}
Proposition~\ref{p:EVPL2control},
the best approximation result of 
Proposition~\ref{p:EVPcomparison} and $M_J\leq \widehat M_J$ imply
\begin{equation*}
 \norm{u-\Lambda_\ell u}
 \leq C_{L^2}(1+\widehat M_J)\hnull^s\norm{(1-\Pi_\ell^0) D^2 u} .
\end{equation*}
The sum of the preceding two displayed formulas concludes
the proof:
Since $\hnull\ll 1$,
the oscillation term 
$\osc_2(\lambda u,\tri_\ell)\lesssim \hnull \norm{u-\Lambda_\ell u}$
 can be absorbed.
\end{proof}

The next result states that the error of the eigenfunction approximation
$\widehat\Lambda_\ell u$ in $\widehat V_\ell$ is comparable with
the best-approximation of the Hessian by piecewise constants.

\begin{proposition}[comparison result for $\widehat \Lambda_\ell$]
                   \label{p:PseudoEVP_BApoiss}
Provided $\hnull\ll 1$, any eigenpair
 $(\lambda,u)\in \mathbb R \times W$ of \eqref{e:ExactBihEVP} 
 with $\norm{u}=1$ satisfies
\begin{equation*}
\ennormnc{(1-\widehat \Lambda_\ell) u}
\lesssim 
\norm{(1-\Pi_\ell^0) D^2 u} .
\end{equation*}
\end{proposition}

\begin{proof}
 The triangle inequality gives
\begin{equation*}
 \ennormnc{(1-\widehat \Lambda_\ell) u}
\leq
\ennormnc{(1-\widehat R_\ell) u}
+
\ennormnc{(\widehat R_\ell -\widehat\Lambda_\ell) u} .
\end{equation*}
Proposition~\ref{p:PseudoBA} implies that the first
term on the right-hand side is controlled by
$\norm{(1-\Pi_0) D^2u)}$.
Set $\hat\varphi_\ell := (\widehat R_\ell -\widehat\Lambda_\ell) u$.
The definition of $\widehat R_\ell$ (note that the right-hand side 
is $f:=\lambda u$) and 
Lemma~\ref{l:LambdaPLemma} 
(with $H_\ell$ replaced by $\widehat V_\ell$) lead to
\begin{equation*}
 \ennormnc{(\widehat R_\ell -\widehat\Lambda_\ell) u}^2
=
a_\nc((\widehat R_\ell -\widehat\Lambda_\ell) u, \hat\varphi_\ell)
=
\lambda b(u-\widehat P_\ell u,\hat\varphi_\ell)
\leq
\lambda \norm{u-\widehat P_\ell u} \,\norm{\hat\varphi_\ell}
.
\end{equation*}
The discrete Friedrichs inequality 
(Corollary~\ref{c:dFMorley}) shows that 
$\norm{\hat\varphi_\ell}\lesssim \ennormnc{\hat\varphi_\ell}$.
The $L^2$ error estimate
from Proposition~\ref{p:L2pesudoevpPoiss}
concludes the proof. Indeed, the resulting higher-order term
$(1+\widehat M_J)\lambda\hnull^s \ennormnc{(1-\widehat \Lambda_\ell) u}$
can be absorbed for $\hnull\ll 1$.
\end{proof}

The tools developed in this section lead to
the following eigenvalue error estimate

\begin{theorem}[eigenvalue error estimates]\label{t:RobustEstimates}
Provided $\hnull\ll 1$, it holds that
\begin{equation*}
\begin{aligned}
\max_{j\in J}
 \frac{\abs{\lambda_j - \lambda_{\ell,j}}}{\max\{\lambda_j,\lambda_{\ell,j}\}}
&\lesssim
 (1+ \widehat M_J^2 B^2)
  \sin^2_{a,\nc}\angle (W,W_\ell)
\\
&\lesssim
 (1+ \widehat M_J^2 B^2)
   \sup_{\substack{w\in W \\ \ennormnc{ w}=1}}
   \norm{(1-\Pi_\ell^0)D^2w}_{L^2(\Omega)}^2 .
\end{aligned}
 \end{equation*}

\end{theorem}

The proof of Theorem~\ref{t:RobustEstimates} 
requires the following Lemma
with the constant $C_{dF}$ from the discrete Friedrichs inequality
of Corollary~\ref{c:dFMorley}.

\begin{lemma} \label{l:EVerrorPois}
The separation condition \eqref{e:separationHat}
from Corollary~\ref{c:separationMhat} implies
\begin{equation*}
 \max_{j\in J}
 \frac{\abs{\lambda_j - \lambda_{\ell,j}}}{\max\{\lambda_j,\lambda_{\ell,j}\}}
\leq 2
(1+\widehat M_J^2 B^2 C_{dF}^4) 
  \left(
  \sin^2_{a,\nc}\angle(W,\widehat W_\ell)
  +
  \sin^2_{a,\nc}\angle(W,W_\ell)
 \right) .
\end{equation*}
\end{lemma}

\begin{proof}
Notice that,
in contrast to the case of conforming finite element methods,
the sign of
$\lambda_j -\lambda_{\ell,j}$ is not known
in the present case of nonconforming methods.

The min-max principle and
Theorem~\ref{t:Knyazev}
(where $H$ is replaced by $\widehat V_\ell$
and $H_\ell$ is replaced by $V$)  prove
\begin{equation} \label{e:PoisNC:KnyOsbEstCaseGE}
 \lambda_j -\lambda_{\ell,j}
\leq
\lambda_j - \hat\lambda_{\ell,j}
\leq
  \lambda_j
  (1+  \widehat M_J^2 B^2  C_{dF}^4)
    \sin^2_{a,\nc}\angle(\widehat W_\ell,W) .
\end{equation}
Here, Theorem~\ref{t:Knyazev} has been applied to the
case that the eigenvalues in $V$ are Ritz values of the eigenvalues
in $\widehat V_\ell$.
Notice carefully that Theorem~\ref{t:Knyazev}
does not require a finite dimension of the ``approximating'' subspace 
(in this case $V$) as pointed out in Remark~\ref{r:Cluster:KOdiminfOK}.

Since the eigenvalue cluster $J$ is finite and, therefore,
the spaces $\widehat W_\ell$ and $W$ have equal finite dimension,
the identity \eqref{e:SinVertausch} implies that
\begin{equation*}%
\sin^2_{a,\nc}\angle(\widehat W_\ell,W)
=
\sin^2_{a,\nc}\angle(W,\widehat W_\ell).
\end{equation*}

In order to bound the modulus $\abs{\lambda_j -\lambda_{\ell,j}}$,
consider also the reverse sign.
Notice that the nonconforming finite element space $V_\ell$ acts as 
a conforming subspace of $\widehat V_\ell$.
The min-max principle and
Theorem~\ref{t:Knyazev} 
(where $H$ is replaced by $\widehat V_\ell$)
then prove
\begin{equation*}
 \lambda_{\ell,j} -\lambda_j
\leq
\lambda_{\ell,j} - \hat\lambda_{\ell,j}
\leq
  \lambda_{\ell,j}
  (1+  \widehat M_J^2 B^2  C_{dF}^4)
    \sin^2_{a,\nc}\angle(\widehat W_\ell,W_\ell) .
\end{equation*}

The formulas \eqref{e:SinVertausch}--\eqref{e:SinDreiecksungl} imply
\begin{equation*}
\begin{split}
\sin^2_{a,\nc}\angle(\widehat W_\ell,W_\ell) \big/ 2
&
\leq
\sin^2_{a,\nc}\angle(\widehat W_\ell,W)
   +\sin^2_{a,\nc}\angle(W,W_\ell)
\\
&
=
\sin^2_{a,\nc}\angle(W,\widehat W_\ell)
   +\sin^2_{a,\nc}\angle(W,W_\ell).
\end{split}
\end{equation*}
\end{proof}

\begin{proof}[Proof of Theorem~\ref{t:RobustEstimates}]
For any $j\in J$, 
Lemma~\ref{l:EVerrorPois} implies
\begin{equation*}
 \frac{\abs{\lambda_j - \lambda_{\ell,j}}}{\max\{\lambda_j,\lambda_{\ell,j}\}}
\leq 2
 (1+ \widehat M_J^2 B^2 C_{dF}^4)
 \left(
 \sin^2_{a,\nc}\angle(W,\widehat W_\ell) 
   +
  \sin^2_{a,\nc}\angle(W,W_\ell) 
 \right).
 \end{equation*}
Proposition~\ref{p:PseudoEVP_BApoiss} shows
\begin{equation*}
\sin^2_{a,\nc}\angle(W,\widehat W_\ell) 
  \lesssim \sin^2_{a,\nc}\angle(W,W_\ell).
\end{equation*}
This proves the first stated inequality. The second inequality
follows from Proposition~\ref{p:EVPcomparison}.
\end{proof}

\begin{remark}
Similar eigenvalue error estimates can be proven for the
nonconforming $\mathcal P_1$ finite element method for the
eigenvalues of the Laplacian or the Stokes operator with the
operators described in \citep{Gallistl2014nc}.
The error estimates of \citep{BoffiDuranGardiniGastaldi2014} for
the eigenvalues of the Laplacian are based on a different methodology.
The authors make use of a conforming $\mathcal P_1$ subspace which makes
a generalisation to the Stokes or the biharmonic eigenvalue problem
appear difficult. On the other hand, they require
less restrictions on the initial mesh-size.
\end{remark}

\section{Adaptive Finite Element Method}\label{s:AFEM}

As an application of the $L^2$ and eigenvalue error estimates developed
in the foregoing sections, this section presents optimal convergence
rates for the adaptive Morley FEM for eigenvalue clusters.

\subsection{Adaptive Algorithm and Optimal Convergence Rates}
\label{ss:AdaptAlg}

This subsection introduces the adaptive algorithm and states 
the optimality result.

For any triangle $T\in\tri_\ell$, the
explicit residual-based error estimator
consists of the
sum of the residuals of the computed discrete eigenfunctions
$(u_{\ell,j})_{j\in J}$,
\begin{equation*}
\begin{aligned}
 \eta_\ell^2(T) :=
  \sum_{j\in J}
     \bigg( h_T^4\norm{\lambda_{\ell,j} u_{\ell,j}}_{L^2(T)}^2 
  &
  +  
       \sum_{F\in\faces(T)\cap\faces_\ell(\Omega\cup\Gamma_C)}
          h_T \norm{[D^2_\nc u_{\ell,j}]_F\tau_F}_{L^2(F)}^2
  \\
  &
  +     \sum_{F\in\faces(T)\cap\faces_\ell(\Gamma_S)}
          h_T \norm{([D^2_\nc u_{\ell,j}]_F\tau_F)\cdot\tau_F}_{L^2(F)}^2\bigg) 
.
\end{aligned}
\end{equation*}
Let, for any subset $\mathcal{K}\subseteq\tri$,
\begin{equation*}
 \eta_\ell^2(\mathcal K) := \sum_{T\in\mathcal K} \eta_\ell^2(T) .
\end{equation*}
This type of error estimator was introduced by
\citet{BeiNiirSten2007,BeiNiirSten2010} and
\citet{HuShi2009} for linear problems.
The methodology to consider the sum of the residuals of the
computed eigenfunctions was first employed in \citep{DaiHeZhou2012v2}
for the case of a multiple eigenvalue.

The adaptive algorithm is driven by this computable error
estimator and runs the following loop.

\begin{algorithm}[AFEM for the biharmonic eigenvalue problem]
  \label{a:AFEM}
\textcolor{white}{.}
\newline
\textbf{Input:} Initial triangulation $\tri_0$, bulk parameter $0<\theta\le 1$.

\noindent
\textbf{for} {$\ell=0,1,2,\dots$}

{\it Solve.}
  Compute discrete eigenpairs 
  $(\lambda_{\ell,j},u_{\ell,j})_{j\in J}$
  of \eqref{e:DiscrBihEVP} with respect to $\tri_\ell$.

{\it Estimate.}
  Compute local contributions of the error estimator
  $\big(\eta_\ell^2(T)\big)_{T\in\tri_\ell}$.

{\it Mark.}
  Choose a minimal subset 
  $\mathcal{M}_\ell\subseteq\tri_\ell$
  such that
  $
    \theta \eta_\ell^2 (\tri_\ell)
  \le  \eta_\ell^2 (\mathcal{M}_\ell) .
 $

{\it Refine.}
  Generate $\tri_{\ell+1}$ from $\tri_\ell$ and $\mathcal M_\ell$
  with newest-vertex bisection \citep{BinevDahmenDeVore2004,Stevenson2008}.

\noindent
\textbf{end for}

\noindent
\textbf{Output:}
 Triangulations $\left(\tri_\ell\right)_\ell$
 and discrete solutions
$\big((\lambda_{\ell,j},u_{\ell,j})_{j\in J}\big)_\ell$.
\end{algorithm}

Let, for any $m\in\mathbb{N}$, the set of triangulations
in $\mathbb T$ whose cardinality differs from that of
$\tri_0$ by $m$ or less  be denoted by
$$\mathbb T(m):=
  \{\tri\in\mathbb T 
        \mid 
    \operatorname{card}(\tri) - \operatorname{card}(\tri_0)
               \leq m \}
               .
$$
Define the seminorm
\begin{equation*}
 \abs{u}_{\mathfrak A_\sigma }
 := \sup_{m\in\mathbb{N}} m^\sigma 
       \inf_{\tri\in\mathbb{T}(m)}
      \norm{(1-\Pi^0_\tri) D^2 u}
\end{equation*}
and the approximation class
\begin{equation*}
 \mathfrak A_\sigma :=
 \left\{ v \in V \bigm|
         \abs{v}_{\mathfrak A_\sigma} < \infty \right\}.
\end{equation*}
The set $\mathfrak A_\sigma$
does not depend on the finite element method
and instead concerns the approximability of the Hessian
by piecewise constant functions.
The following alternative set, also referred to as 
approximation class, is employed in the analysis of
the optimal convergence rates
\begin{equation*}
\mathfrak A_\sigma^{\mathrm{Morley}} :=
 \left\{ u \in V \bigm|
         \abs{u}_{\mathfrak A_\sigma^{\mathrm{Morley}}} < \infty \right\}
\end{equation*}
for
\begin{equation*}
  \abs{u}_{\mathfrak A_\sigma^{\mathrm{Morley}}}
 := \sup_{m\in\mathbb{N}} m^\sigma 
       \inf_{\tri\in\mathbb{T}(m)}
       \ennorm{ u - \Lambda_\tri u} .
\end{equation*}
Proposition~\ref{p:EVPcomparison} establishes
the equivalence of those two approximation classes in
the sense that any eigenfunction $u\in W$ satisfies
$u\in\mathfrak A_\sigma$ if and only if
$u \in\mathfrak A_\sigma^{\mathrm{Morley}}$.
The following theorem states optimality of Algorithm~\ref{a:AFEM}.
The proof will be outlined throughout the remaining parts of this 
section.

\begin{theorem}[optimal convergence rates]\label{t:optimality}
 Let $\Omega$ be simply-connected.
 Provided the bulk parameter $\theta\ll 1$ and the initial mesh-size
 $\hnull\ll1 $ are sufficiently small,
 Algorithm~\ref{a:AFEM} computes triangulations $(\tri_\ell)_\ell$
 and discrete eigenpairs
 $\left( (\lambda_{\ell,j}, u_{\ell,j})_{j\in J} \right)_{\ell}$
 with optimal rate of convergence in the sense that, for
 some constant $C_{\mathrm{opt}}$,
 \begin{equation*}
 \sup_{\ell\in\mathbb N}
    \big(\card (\tri_\ell) 
                - \card (\tri_0)\big)^{\sigma} 
  \left(  \sum_{j\in J} 
        \ennormnc{u_j - \Lambda_\ell u_j}^2
   \right)^{1/2}
   \leq 
    C_{\mathrm{opt}}
   \left(\sum_{j\in J} \abs{u_j}_{\mathfrak A_\sigma^{\mathrm{Morley}}}^2\right)^{1/2}
   .
 \end{equation*}

\end{theorem}

Proposition~\ref{p:EVPcomparison}, Theorem~\ref{t:RobustEstimates}
and
Remark~\ref{r:normalis} immediately imply the following
consequence.

\begin{corollary}\label{c:Optim}
Let $\Omega$ be simply-connected.
 Provided the bulk parameter $\theta\ll 1$ and the initial mesh-size
 $\hnull\ll1$ are sufficiently small,
 Algorithm~\ref{a:AFEM} computes triangulations $(\tri_\ell)_\ell$
 and discrete eigenpairs
 $\left( (\lambda_{\ell,j}, u_{\ell,j})_{j\in J} \right)_{\ell}$
 with optimal rate of convergence in the sense that
  \begin{align*}
  &
   (1+\widehat M_J^2 B^2)^{-1/2}
  \max_{k\in J}
   \left(\frac{\abs{\lambda_k - \lambda_{\ell,k}}}
              {\max\{\lambda_k,\lambda_{\ell,k}\}}\right)^{1/2}
   +
   \sin_{a,\nc}\angle(W,W_\ell)
   \\
   &
    \qquad\qquad\qquad
   \lesssim
   A^{-1/2}
   (\operatorname{card}(\tri_\ell) - \operatorname{card}(\tri_0))^{-\sigma} 
   \left(\sum_{j\in J} \abs{u_j}_{\mathfrak A_\sigma}^2\right)^{1/2}
   .
   \qed
  \end{align*}
\end{corollary}

\subsection{Discrete Reliability}
        \label{s:drel}

This section generalises the discrete Helmholtz decomposition 
from \citep{CarstensenGallistlHu2014} to more general boundary conditions.
The decomposition can be viewed as a discrete analogue of
\citep[Lemma~1 and Corollary~1]{BeiNiirSten2010}.

Define
\begin{equation*}
 \hat H^1(\Omega;\mathbb R^2)
 :=
 \left\{v \in H^1(\Omega;\mathbb R^2)
     \bigm| \textstyle\int_\Omega v \,dx = 0 \text{ and } 
       \textstyle\int_\Omega \operatorname{div} v \,dx = 0 \right\} 
\end{equation*}
and
\begin{equation*}
\mathfrak X(\tri_\ell):= 
\left\{ 
 \begin{array}{l}
 v\in \mathcal P_1(\tri_\ell;\mathbb{R}^2) \\
 \phantom{v\in}\cap \hat H^1(\Omega;\mathbb{R}^2)
 \end{array}
\left|
 \begin{array}{l}
 \text{1. for all } 
   F=\operatorname{conv}\{z_1,z_2\}\in\faces_\ell(\Gamma_S\cup\Gamma_F)\\ 
   \quad(v(z_2)-v(z_1))\cdot\nu_F =0, \\
 \text{2. for all } (F_-,F_+)\in\faces_\ell(\Gamma_F)^2 \\
   \quad \text{with } F_- = \operatorname{conv}\{z_-,z\},
                    F_+ = \operatorname{conv}\{z,z_+\}\\
 \quad h_{F_-}^{-1} (v(z)-v(z_-))\cdot\tau_{F_-}
       = 
       h_{F_+}^{-1} (v(z_+)-v(z))\cdot\tau_{F_+}
 \end{array}
\right.
\right\}   .
\end{equation*}

\begin{remark}
In other words, the functions of $\mathfrak X(\tri_\ell)$
satisfy that
$\partial(\psi\cdot\nu)/\partial\tau = 0$ on $\Gamma_S \cup\Gamma_F$
and
$(D\psi \tau)\cdot\tau$ is constant on each connectivity component
of $\Gamma_F$.
The definition of $\mathfrak X(\tri_\ell)$ above is
stated in such a way that one can see that this defines
$\card(\faces_\ell(\Gamma_S\cup\Gamma_F)) + \card(\mathcal N_\ell(\Gamma_F))$
linear independent contraints on 
$P_1(\tri_\ell;\mathbb{R}^2)\cap \hat H^1(\Omega;\mathbb{R}^2)$.
Recall that $\Gamma_C$ and $\Gamma_C\cup\Gamma_S$ are
assumed to be closed sets and, thus,
$\mathcal N_\ell(\Gamma_F)$ contains exactly those vertices
that are shared by two edges of $\Gamma_F$.
\end{remark}

\begin{theorem}[discrete Helmholtz decomposition for
                piecewise constant symmetric tensor fields]
                \label{t:discretehelmholtzsym}
 Let $\Omega$ be simply-connected.
 Given any piecewise constant symmetric tensor field
 $\sigma_\ell \in \mathcal P_0(\tri_\ell;\mathbb{S})$,
 there exist unique $\phi_\ell \in V_\ell$
 and $\psi_\ell \in \mathfrak X(\tri_\ell)$ such that
 \begin{equation}\label{e:dHelmholtzSymDecomp}
  \sigma_\ell = D^2_\nc \phi_\ell  + \operatorname{sym}\Curl\psi_\ell.
 \end{equation}
 The decomposition is $L^2$ orthogonal and 
 the functions $\phi_\ell$, $\psi_\ell$, $\sigma_\ell$ from
 \eqref{e:dHelmholtzSymDecomp} satisfy
 \begin{equation}\label{e:dHelmholtzSymStab}
  \norm{D^2_\nc \phi_\ell}_{L^2(\Omega)} 
     + \lVert \Curl \psi_\ell \rVert_{L^2(\Omega)}
   \lesssim 
      \lVert \sigma_\ell \rVert_{L^2(\Omega)}.
 \end{equation}
 \end{theorem}
 
 \begin{proof}
Since the contributions on the right-hand side of
\eqref{e:dHelmholtzSymDecomp} are $L^2$-orthogonal
and since
\begin{equation*}
  D^2_\nc (V_\ell) + \operatorname{sym}\Curl(\mathfrak X(\tri_\ell))
 \subseteq\mathcal P_0(\tri_\ell;\mathbb S),
\end{equation*}
it suffices to prove
\begin{equation*}
  \dim(\mathcal P_0(\tri_\ell;\mathbb{S}))
  =
  \dim(D^2_\nc (V_\ell)) 
   + \dim(\operatorname{sym}\Curl(\mathfrak X(\tri_\ell))).
\end{equation*}
The proof of this formula
follows from the well-known Euler formulae (for two space dimensions
and simply-connected domains;
the proof follows from mathematical induction)
\begin{equation*}
  \card(\mathcal{N}_\ell) + \card(\tri_\ell) = 1 + \card(\faces_\ell)
  \quad\text{and}\quad
  2\,\card(\tri_\ell) + 1
          =  \card(\mathcal{N}_\ell) +\card(\faces_\ell(\Omega) ) .
\end{equation*}
The proof of the stability \eqref{e:dHelmholtzSymStab} is proven
in \citep[Lemma~3.3]{CarstensenGallistlHu2014}.
 \end{proof}

The remaining parts of this subsection prove the discrete reliability
for a theoretical error estimator.
The idea to include such a non-computable quantity in the analysis
of adaptive algorithms was first introduced in
\citep{DaiHeZhou2012v2} in the context of multiple eigenvalues.
The theoretical error estimator does not depend on
the choice of the discrete eigenfunctions.
Given an eigenpair $(\lambda,u)$,
the error estimator is defined, for any $T\in\tri_\ell$, as
\begin{equation*}
\begin{aligned}
 \mu_\ell^2(T,\lambda,u) :=
  \sum_{j\in J}
     \bigg( h_T^4\norm{\lambda P_\ell u}_{L^2(T)}^2 
  &
  +  
       \sum_{F\in\faces(T)\cap\faces_\ell(\Omega\cup\Gamma_C)}
          h_T \norm{[D^2_\nc \Lambda_\ell u]_F\tau_F}_{L^2(F)}^2
  \\
  &
  +     \sum_{F\in\faces(T)\cap\faces_\ell(\Gamma_S)}
          h_T \norm{([D^2_\nc \Lambda_\ell u]_F\tau_F)\cdot\tau_F}_{L^2(F)}^2\bigg) 
.
\end{aligned}
\end{equation*}
Define, for any subset $\mathcal{K}\subseteq\tri_\ell$,
\begin{equation*}
 \mu_\ell^2(\mathcal{K},\lambda_j,u_j) :=
 \sum_{T\in\mathcal{K}} \mu_\ell^2(T,\lambda_j,u_j)
 \quad\text{and}\quad
 \mu_\ell^2(\mathcal{K}) := \sum_{j\in J} \mu_\ell^2(T,\lambda_j,u_j) .
\end{equation*}

The following shorthand notation for higher-order terms
with respect to an eigenpair $(\lambda,u)\in\mathbb R\times W$
of \eqref{e:ExactBihEVP}
is employed throughout this section
\begin{equation}\label{e:hotDef}
 \bm r_{\ell,m}:=
 \hnull^{s} \lambda (1+M_J) C_{L^2}
        \sqrt{\ennorm{u - \Lambda_\ell u}^2
         +  \ennorm{u - \Lambda_{\ell+m} u}^2} .
\end{equation}

The following Lemma carefully explores the properties of the
quasi-interpolation of \citet{ScottZhang1990}.

\begin{lemma}[Scott-Zhang quasi-interpolation]\label{l:ScottZhang}
Let $\tri_{\ell+m}$ be a refinement of
$\tri_\ell$ and let 
$\psi_{\ell+m}\in \mathcal P_1(\tri_{\ell+m};\mathbb R^2) \cap H^1(\Omega;\mathbb R^2)$
be such that
$(D\psi_{\ell+m} \tau)\cdot\nu = 0$ on $\Gamma_S \cup\Gamma_F$
and
$(D\psi_{\ell+m} \tau)\cdot\tau$ is constant on each connectivity component
of $\Gamma_F$.
Then there exists 
$\psi_\ell \in \mathcal P_1(\tri_\ell;\mathbb R^2)\cap H^1(\Omega;\mathbb R^2)$
with the property that $\psi_\ell|_F = \psi_{\ell+m}|_F$ for all
edges $F \in\faces_\ell \cap \faces_{\ell+m}$.
Moreover, the function $\psi_\ell$ can be chosen in such a way
that it preserves the boundary conditions in the sense that
$(D\psi_\ell \tau)\cdot\nu = 0$ on $\Gamma_S \cup\Gamma_F$
and
$(D\psi_\ell \tau)\cdot\tau$ is constant on each connectivity component
of $\Gamma_F$.
This quasi-interpolation satisfies the approximation and stability
estimate
\begin{equation*}
 \norm{h_\ell^{-1} (\psi_{\ell+m}-\psi_\ell)}_{L^2(\Omega)}
 +
  \norm{D (\psi_{\ell+m}-\psi_\ell)}_{L^2(\Omega)}
\lesssim
 \norm{D \psi_{\ell+m}}_{L^2(\Omega)} .
\end{equation*}
\end{lemma}

\begin{remark}
 The quasi-interpolation of Lemma~\ref{l:ScottZhang} preserves
 the boundary conditions imposed on the space
 $\mathfrak X(\tri_{\ell+m})$
 for any refinement $\tri_{\ell+m}$.
\end{remark}

\begin{proof}[Proof of Lemma~\ref{l:ScottZhang}]
 The methodology of \citet{ScottZhang1990} assigns to each vertex
 $z\in\mathcal N_\ell$ some edge $F_z\in\faces_\ell$.
 The choice assigns, whenever possible, to a vertex $z\in\mathcal N_\ell$
 an edge $F_z\in\faces_\ell\cap\faces_{\ell+m}$.
 For vertices $z\in\overline \Gamma_F$ that touch the free boundary,
 choose $F_z\in\faces_\ell(\Gamma_F)$ if this does not contradict
 a possible choice of $F_z\in\faces_\ell\cap\faces_{\ell+m}$ .
 Let, for any edge $F_z\in\faces_\ell$, $\Phi_z\in L^2(F_z)$ denote the
 Riesz representation of the point evaluation $\delta_z$ at $z$
 in the space $\mathcal P_1(F)$.
 
 For vertices that touch the simply supported part of the boundary
 but not the free part 
 $z\in\overline\Gamma_S \setminus\overline \Gamma_F$
 and that do not belong to any edge of $\faces_\ell\cap \faces_{\ell+m}$,
 denote the adjacent boundary edges by  $(F_1,F_2)\in\faces_{\ell}^2$
 and define
 \begin{equation*}
   \nu_{F_1}\cdot\psi_\ell(z)
       = \int_{F_1} \Phi_z \nu_{F_1}\cdot\psi_{\ell+m} \,ds
\quad\text{and}\quad
   \nu_{F_2}\cdot\psi_\ell(z)
       = \int_{F_2} \Phi_z \nu_{F_2}\cdot\psi_{\ell+m}\,ds  .
 \end{equation*}
If the angle between $F_1$ and $F_2$ equals $\pi$,
then $\nu_{F_1} =\nu_{F_2}$ and this definition is
consistent.
In this case set    
$\tau_{F_1}\cdot\psi_\ell(z)
       = \int_{F_1} \Phi_z \tau_{F_1}\cdot\psi_{\ell+m} \,ds$.
For all remaining vertices $z$ of $\tri_\ell$,
 define
 $\psi_\ell (z) \cdot e_j := \int_{F_z} \Phi_z \psi_{\ell+m}\cdot e_j \,ds$
 for the unit vectors $e_j\in\{(1;0) , (0;1)\}$.
 
This definition of $\psi_\ell$ is an admissible choice in the setting
of \citet{ScottZhang1990}.
In particular, $\psi_\ell$ coincides with $\psi_{\ell+m}$ on edges of 
$\faces_\ell\cap\faces_{\ell+m}$. The error estimate follows from
the theory in \citep{ScottZhang1990}.

It remains to show the claimed boundary conditions.
Recall that $\psi_{\ell+m}$
satisfies 
$(D\psi_{\ell+m} \tau)\cdot\nu = 0$ on $\Gamma_S \cup\Gamma_F$
and
$(D\psi_{\ell+m} \tau)\cdot\tau$ is constant on each connectivity component
of $\Gamma_F$.
In particular, this implies that
$\psi_{\ell+m}\cdot\nu$ is constant along each straight part of
$\Gamma_S \cup\Gamma_F$ and that $\psi_{\ell+m}\cdot \tau$ is
affine along each straight part of $\Gamma_F$.
Therefore, the above assignment of the nodal values
interpolates $\psi_{\ell+m}\cdot\nu$ along $\overline{\Gamma_S \cup\Gamma_F}$
and $\psi_{\ell+m}\cdot \tau$ along $\overline{\Gamma_F}$ exactly and
so these boundary conditions are valid for $\psi_\ell$.
\end{proof}

The next proposition states the discrete reliability. The idea
to prove such type of result by means of a discrete Helmholtz
decomposition was first employed in \citep{BeckerMaoShi2010}
for the Poisson equation.

\begin{proposition}[discrete reliability]\label{p:drel}
 There exists a constant $C_{\mathrm{drel}} \approx 1$ such that,
 for $\hnull\ll 1$,
 any admissible refinement $\tri_{\ell+m} \in \mathbb{T}(\tri_\ell)$ 
 of $\tri_\ell \in \mathbb{T}$ and any eigenpair
 $(\lambda,u)\in\mathbb R \times W$ 
 of \eqref{e:ExactBihEVP}
 with $\norm{u}=1$ 
and $\bm r_{\ell,m}$ from \eqref{e:hotDef} satisfy
 \begin{equation*}
  2 \ennormnc{(\Lambda_{\ell+m} - \Lambda_\ell) u}^2
    \leq C_{\mathrm{drel}}^2 (
    \mu_\ell^2(\tri_\ell \setminus \tri_{\ell+m})
   + \bm r_{\ell,m}^2).
 \end{equation*}
\end{proposition}
 
 \begin{proof}
  The discrete Helmholtz decomposition from 
  Theorem~\ref{t:discretehelmholtzsym} leads to 
  $\phi_{\ell+m} \in V_{\ell+m}$ and
  $\psi_{\ell+m} \in \mathfrak X(\tri_{\ell+m})$ such that
  \begin{equation*}
    D^2_\nc ((\Lambda_{\ell+m}-\Lambda_\ell)u)
    = D^2_\nc \phi_{\ell+m} + \operatorname{sym}\Curl \psi_{\ell+m}.
  \end{equation*}
  The orthogonality of the decomposition proves
  \begin{equation}\label{e:drelsplit}
    \ennormnc{(\Lambda_{\ell+m}-\Lambda_\ell)u}^2
    =
    a_\nc((\Lambda_{\ell+m}-\Lambda_\ell)u, \phi_{\ell+m})
    -
    (D^2_\nc \Lambda_\ell u ,\Curl \psi_{\ell+m})_{L^2(\Omega)}.
  \end{equation}

  The projection property of the Morley interpolation
  operator \eqref{e:MorleyInterpolProjProp},
  Lemma~\ref{l:LambdaPLemma}, 
  the $L^2$ control of Proposition~\ref{p:EVPL2control}
  and the approximation
  and stability property \eqref{e:MorleyInterpolApprxStab}
  prove for the first term of \eqref{e:drelsplit} that
  \begin{equation*}
   \begin{aligned}
    a_\nc( (\Lambda_{\ell+m} - \Lambda_\ell)u, \phi_{\ell+m})
   & 
   = \lambda b( (P_{\ell+m}-P_\ell) u, \phi_{\ell+m} )
      + \lambda b(P_\ell u, (1-\Imorl_\ell) \phi_{\ell+m} )
   \\
   &
    \lesssim
    (\bm r_{\ell,m}
     +
    \lVert h_\ell^2 \lambda P_\ell u\rVert_{L^2(\cup(\tri_\ell\setminus \tri_{\ell+m}))} 
    )
     \ennormnc{\phi_{\ell+m}}.
  \end{aligned}
  \end{equation*}
  Let $\psi_\ell \in \mathcal P_1(\tri_\ell;\mathbb R^2)\cap H^1(\Omega;\mathbb R^2)$
  denote the quasi-interpolation from Lemma~\ref{l:ScottZhang}.
  The function $\psi_\ell$ preserves those boundary conditions of
  $\psi_{\ell+m}$ that are necessary to guarantee that
  $\Curl\psi_\ell$ and $D^2_\nc \Lambda_\ell u$ are $L^2$-orthogonal.
  Hence, an integration by parts shows for the second term of
  \eqref{e:drelsplit} that
  \begin{equation*}
    (D^2_\nc \Lambda_\ell u , \Curl \psi_{\ell+m})_{L^2(\Omega)}
    =
    \sum_{F\in\faces_\ell \setminus \faces_{\ell+m}}
     \int_F ([D^2_\nc \Lambda_\ell u]_F \tau_F)
             \cdot (\psi_{\ell+m}-\psi_\ell)\,ds.
  \end{equation*}
  The boundary conditions of $\psi_{\ell+m}$ and $\psi_\ell$
  plus
  Cauchy and trace inequalities and the approximation
  and stability properties of the Scott-Zhang
  quasi-interpolation prove that this is bounded by
  $\lVert D \psi_{\ell+m} \rVert_{L^2(\Omega)} $ times
  \begin{equation*}
    \left(
    \sum_{T\in\tri_\ell\setminus\tri_{\ell+m}}
    \left(
    \sum_{\substack{F\in\faces(T)
                     \\ \cap\faces_\ell(\Omega\cup\Gamma_C)}}
     h_F \lVert [D^2_\nc \Lambda_\ell u]_F \tau_F \rVert_{L^2(F)}^2
     +
     \sum_{\substack{F\in\faces(T) \\
                   \cap\faces_\ell(\Gamma_S)}}
     h_F \lVert \tau_F\cdot ([D^2_\nc \Lambda_\ell u]_F \tau_F )\rVert_{L^2(F)}^2
     \right)
     \right)^{1/2} .
  \end{equation*}
  The combination of the foregoing estimates and 
  the stability \eqref{e:dHelmholtzSymStab} conclude the proof.
 \end{proof}

The following reliability and efficiency are an immediate consequence 
of the discrete reliability and a~priori convergence results
(e.g., Proposition~\ref{p:EVPcomparison}).

\begin{corollary}[reliability and efficiency]
 \label{c:RelEff}
 Provided $\hnull\ll1$, it holds that
\begin{equation*}\label{e:reliability}
 \ennormnc{u-\Lambda_\ell u}^2
\lesssim \mu_\ell^2(\tri_\ell,\lambda,u)
\lesssim \ennormnc{u-\Lambda_\ell u}^2  .
\end{equation*}
\qed
\end{corollary}

\subsection{Proof of Optimal Convergence Rates}

The proof of the discrete reliability is the main step
in proving optimal convergence rates 
for Algorithm~\ref{a:AFEM}. 
Proofs for optimal convergence rates of the D\"orfler
marking strategy \citep{Doerfler1996} are mainly based
on the ideas of \citet{Stevenson2007} and \citet{CKNS08} and were recently
unified in the axiomatic framework of \citet{CFPP2013}.
Hence, the remaining arguments are not carried out in detail here
but only sketched with references to similar proofs in the literature.

The quasi-orthogonality for the Morley FEM was first proven by
\citet{HuShiXu2012Morley} in the context of the linear biharmonic
problem. The following result is an extension to the case of
eigenvalue problems.

 \begin{proposition}[quasi-orthogonality]
       \label{p:quasiorth}
 Under the hypothesis $\hnull\ll 1$
 there exists a constant $C_{\mathrm{qo}}$ such that
any eigenpair $(\lambda,u)\in\mathbb{R}\times W$ 
of \eqref{e:ExactBihEVP}  with $\norm{u}=1$,
 any $\tri_\ell\in\mathbb{T}$ and any admissible refinement
$\tri_{\ell+m}\in\mathbb{T}(\tri_\ell)$ satisfy
 \begin{equation*}
  \begin{aligned}
   &
   \abs{ 2 a_\nc (u - \Lambda_{\ell+m} u,
           \Lambda_{\ell+m} u -\Lambda_\ell u) }
   \\
   &\qquad\qquad
    \leq 
    C_{\operatorname{qo}}    
    \big(
   \norm{h_\ell^2\lambda P_\ell u}_{L^2(\cup\tri_\ell \setminus \tri_{\ell+m})}
   + \bm r_{\ell,m}\big) \ennormnc{u-\Lambda_{\ell+m}u}    .
  \end{aligned}
 \end{equation*}
\end{proposition}
\begin{proof}
The properties of the operator
$\Imorl_\ell$ of Section~\ref{s:MorleyFEM}
together with the arguments of \citet{HuShiXu2012Morley} and
\citet{Gallistl2014nc} lead to the proof.
In particular the constant of Proposition~\ref{p:IlEstimate}
(which is independent of $\tri_{\ell+m}$) enters the analysis.
The details are omitted.
\end{proof}

The following result states an equivalence of the theoretical
error estimator $\mu_\ell$ with the practical error estimator 
$\eta_\ell$.
\begin{proposition}[bulk criterion] \label{p::bulk}
 Suppose that $\hnull\ll1$ satisfies \eqref{e:separationMorley} and 
 \begin{equation*}
 \varepsilon 
 := \max_{j\in J} \norm{u_j - \Lambda_\ell u_j}_{b,\nc} 
 \leq \sqrt{1 + 1/(2N)} -1
 \quad\text{for all } \tri_\ell \in\mathbb T.
 \end{equation*}
 Then, for any $T\in\tri_\ell$, the error estimator contributions
 can be compared as follows
 \begin{equation*}
     N^{-1} \sum_{j\in J} \mu_\ell^2(T,\lambda_j,u_j)
   \leq
     (B/A)^2\eta_\ell^2(T)
   \leq
     (B/A)^4  (2N + 4N^2)
         \sum_{j\in J} \mu_\ell^2(T,\lambda_j,u_j)
.
 \end{equation*}
Therefore,
$
 \mu_\ell(\mathcal{M}_\ell)
   :=\sum_{T\in\mathcal M_\ell}\sum_{j\in J}
     \mu_\ell^2(T,\lambda_j,u_j)
$
 satisfies the bulk
criterion
\begin{equation*}
  \tilde\theta \mu_\ell(\tri_\ell)
  \leq \mu_\ell(\mathcal{M}_\ell)
\end{equation*}
for the modified bulk parameter
\begin{equation}\label{e:modbulk}
\tilde\theta:= 
\left((B/A)^4(2N^2+4N^3)\right)^{-1}\, \theta<1
.
\end{equation}

\end{proposition}
\begin{proof}
 The proof follows from Lemma~5.1 and Proposition~5.2 
 of \citep{Gallistl2014}.
\end{proof}

\begin{proposition}[error estimator reduction for $\mu_\ell$]
  Provided $\hnull\ll 1$, there exist constants 
  $0<\rho_1 < 1$ and $0<K<\infty$ such that 
  $\tri_\ell$ and its one-level refinement $\tri_{\ell+1}$
  generated by Algorithm~\ref{a:AFEM}
  and any eigenfunction $u\in W$ with $\norm{u}=1$ and
  eigenvalue $\lambda$ satisfy (with $\bm r_{\ell,1}$ from
  \eqref{e:hotDef}) that
 \begin{equation*}
  \mu_{\ell+1}^2(\tri_{\ell+1},\lambda,u)
  \leq
  \rho_1 \mu_{\ell}^2(\tri_\ell,\lambda,u)
 +
 K \left( \ennormnc{\Lambda_{\ell+1} u - \Lambda_{\ell} u}^2
 +
  \hnull^4 \bm r_{\ell,1}^2 \right) .
 \end{equation*}
\end{proposition}
\begin{proof}
 The standard techniques of \citep{CKNS08,Stevenson2007}
 and the bulk criterion \eqref{e:modbulk} lead to a constant
 $\tilde K$ such that
 \begin{equation*}
  \begin{aligned}
  &
  \mu_{\ell+1}^2(\tri_{\ell+1},\lambda,u)
  \\
  &
  \quad
  \leq
  \rho_1 \mu_{\ell}^2(\tri_\ell,\lambda,u)
 +
 \tilde K \left( \ennormnc{\Lambda_{\ell+1} u - \Lambda_{\ell} u}^2
 +
  \norm{h_{\ell+1}^2 \lambda (P_{\ell+1}-P_\ell) u}^2\right) .
  \end{aligned}
 \end{equation*}
 The triangle inequality for the term 
 $\norm{h_{\ell+1}^2 \lambda (P_{\ell+1}-P_\ell) u}$
 and the $L^2$ error control
 from Proposition~\ref{p:EVPL2control} prove the result.
\end{proof}

\begin{proposition}[contraction property]\label{p:contraction}
 Under the condition $\hnull\ll 1$,
 there exist $0<\rho_2<1$ and 
 $0<\beta,\gamma<\infty$ such that, for any eigenpair 
 $(\lambda,u)\in \mathbb{R}\times W$ with
 $\norm{u}=1$, the term
 $\xi_\ell^2:= \mu_\ell^2(\tri_\ell,\lambda,u)
                    + \beta \ennormnc{u -\Lambda_\ell u}^2
                    + \gamma \norm{h_\ell^2 P_\ell u}^2$
 satisfies
 \begin{equation*}
   \xi_{\ell+1}^2 \leq \rho_2 \xi_{\ell} ^2
        \quad \text{for all } \ell\in\mathbb{N}_0.
 \end{equation*}
\end{proposition}
\begin{proof}
 The proof is analogous to the proof of contraction in
 \citep{Gallistl2014nc}. The details are omitted.
\end{proof}

The proof of Theorem~\ref{t:optimality} follows
with the preceding four propositions and the discrete
reliability (Proposition~\ref{p:drel}) and is based on the
techniques of \citep{CKNS08,Stevenson2007}.
A similar proof for second-order problems was carried out in
detail in \citep[Sect.~5.5]{Gallistl2014nc} and the proof of
Theorem~\ref{t:optimality} is almost identical.
Further details are omitted here for brevity.

\end{document}